\documentclass[11pt]{amsart}
\usepackage[onehalfspacing]{setspace}

\usepackage{amsmath,amsfonts, amsthm, amssymb}
\usepackage[all,cmtip]{xy}
\usepackage{enumerate}
\usepackage{graphicx}
\usepackage{subfig}
\usepackage{ amssymb, latexsym, amsmath}
\usepackage{fullpage}
\usepackage{url}
\usepackage{hyperref}
\usepackage{ marvosym }
\usepackage{comment}
\usepackage{graphicx}
\nocite{*}
\usepackage{biblatex}
\numberwithin{equation}{section}

\theoremstyle{plain}
\newtheorem{theorem}{Theorem}[section]

\newtheorem{proposition}[theorem]{Proposition}
\newtheorem{lemma}[theorem]{Lemma}

\theoremstyle{definition}
\newtheorem{definition}[theorem]{Definition}

\newtheorem{question}[theorem]{Question}

\theoremstyle{remark}
\newtheorem{remark}[theorem]{Remark}

\newtheorem{claim}[theorem]{Claim}

\makeatletter

\newcommand{\round}[1]{\left\lfloor #1 \right\rfloor_N}
\newcommand{\Rmnum}[1]{\expandafter\@slowromancap\romannumeral #1@}
\makeatother

\newcommand{\Z}{\mathbb{Z}}

\newcommand{\E}{\mathbb{E}}

\newcommand{\N}{\mathbb{N}}

\newcommand{\eps}{\epsilon}

\newcommand{\R}{\mathbb{R}}

\renewcommand{\P}{\mathbb{P}}

\newcommand{\supp}{\operatorname{supp}}

\newcommand{\sgn}{\operatorname{sgn}}

\DeclareMathOperator{\Cov}{Cov}

\title{Long Range Voter Models and Dynamical Fractional Brownian Motion}

\author{Reuben Drogin}
\address{Department of Mathematics, Yale University, New Haven, CT 06511}
\email{reuben.drogin@yale.edu}
\begin{document}

\begin{abstract}
    We study the voter model on $\Z$ with long-range interactions, 
    as proposed by Hammond and Sheffield in \cite{HS}. We show a
    spacetime rescaling converges to a fractional Gaussian free field, which 
    can be viewed as a one-parameter family of
    fractional Brownian motions. As a consequence, we obtain that
    long-range voter models rescale to fractional Gaussian noise. The argument uses 
    the Lindeberg swapping technique and heat kernel estimates for 
    random walks with jump distributions in the domain 
    of attraction of a stable law.  
\end{abstract}
\maketitle
\section{Introduction}

\subsection{Summary and Background}
Consider the following simple interacting particle system on $\mathbb{Z}^d$. At time $0$
each site is assigned the value $+1$ or $-1$, and at each time step the site $n$ updates its value by adopting
the value at $n-k_n$, where $(k_{n})_{n\in\mathbb{Z}^d}$ are i.i.d. samples of some fixed measure 
$\mu$ on $\mathbb{Z}^d$, sampled independently at each time step. 
This is the \textit{voter model with interactions $\mu$}, introduced by Holley and Liggett in \cite{holllig}.
When $\mu$ is short-range, say nearest neighbor, 
there is a rather complete understanding of this system which we now recall.

In $d=1,2$, any two sites will eventually have the same value almost surely. 
This is because the probability that $n,m\in\mathbb{Z}^d$ have the same value at time $t$ is 
at least as big as the probability that random walks started at $n$ and $m$ intersect before time $t$, see \cite{holllig}. 
Intuitively, one might expect to see clusters of aligned values which grow in time. 
Arratia \cite{Arr} and Bramson and Griffeath \cite{clustering}
showed this is indeed the case in $d=1$ with the model at time $t$ consisting mostly of strings of $+1$s or $-1$s with length 
of order $\sqrt{t}$. See figure \ref{fig: spacetime plot} below for an illustration. Cox and Griffeath 
\cite{2dclustering} showed similar behavior in $d=2$.

In $d\geq 3$, the transience of the simple random walk allows for equilibrium states where $+1$s and $-1$s coexist.
Moreover, for suitable initializations the Gaussian free field appears as the scaling limit.
Precisely, if $X(n,t)$ denotes the $\pm 1$ value at site $n$ at time $t$,
then $(X(\cdot,t))_{t\in\mathbb{N}}$ is a Markov chain on the state space $\{-1,+1\}^{\mathbb{Z}^d}$, 
and Holley and Liggett \cite{holllig} showed that if the random walk generated by $\mu$ is transient, 
this Markov chain has a one-parameter family $\{ \nu_p\ |\ p\in[0,1]\}$ of extremal 
invariant measures on $\{-1,+1\}^{\mathbb{Z}^d}$. These $\nu_p$ are translation invariant,
satisfy $\nu_p(X(n,0) = 1) = p$, and have decaying correlations.
Bramson and Griffeath \cite{BG} showed $\nu_p$ rescales to the Gaussian free field when $p\in (0,1)$ and $d=3$.
Zahle \cite{Z} later extended this result to $d\geq 4$.

In this paper, we study the voter model when $\mu$ is a \textit{long-range} measure on $\mathbb{Z}$,
roughly meaning $\mu(\{ \pm m, \pm (m+1),...\})\sim m^{-\alpha}$ for some 
$\alpha\in(0,1)$ (see 
definition \ref{def: measure class} for the precise measure class). For such $\mu$ the random walk generated by $\mu$ is transient 
and thus the corresponding voter model 
has a one-parameter family of extremal invariant measures $\{\nu_p\ | p\in[0,1]\}$. Theorem \ref{cor1} shows these 
measures scale to fractional Gaussian noise with parameter $\alpha/2$, which is white noise convolved with a long-range kernel. 
See figure \ref{fig: spacetime plot} below for an illustration. This complements the classical results of Bramson and Griffeath and Zahle in the short-range case. 
Our proof takes a different approach than \cite{BG} and \cite{Z}, using the Lindeberg swapping method, and in fact
this result is corollary of our main result, Theorem \ref{thm: main}.
\begin{figure}[htp]
    \centering
    \subfloat[\centering A simulation of the voter model for a long range $\mu\in \Gamma_{1/2}$. It exhibits coexistence.]{{
    \includegraphics[width = 7cm, height=6cm]{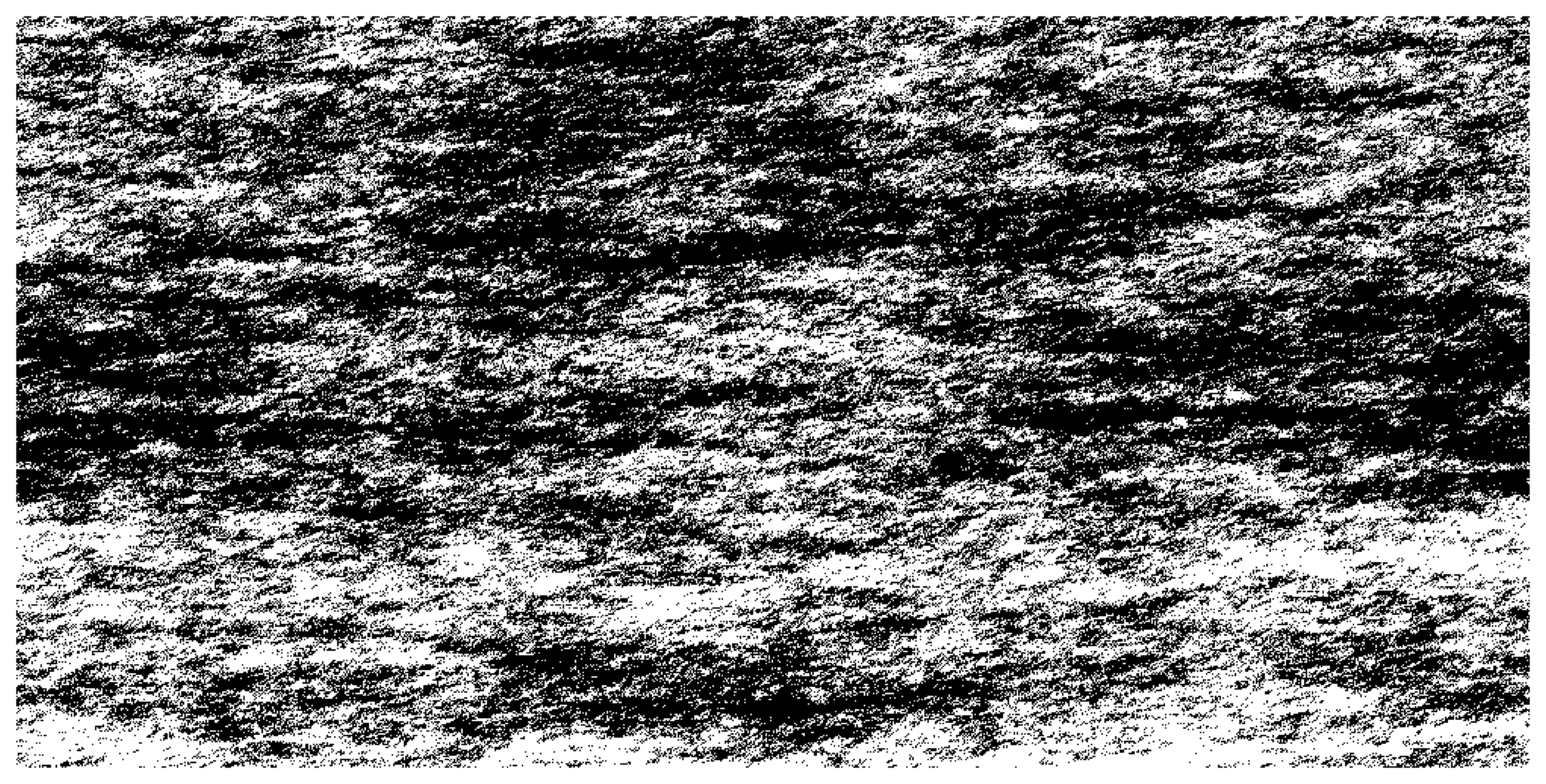}
    }}
    \hspace{1cm}
    \subfloat[\centering A simulation of the voter model when $\mu$ is nearest neighbor. It exhibits clustering]{{
    \includegraphics[width = 7cm, height=6cm]{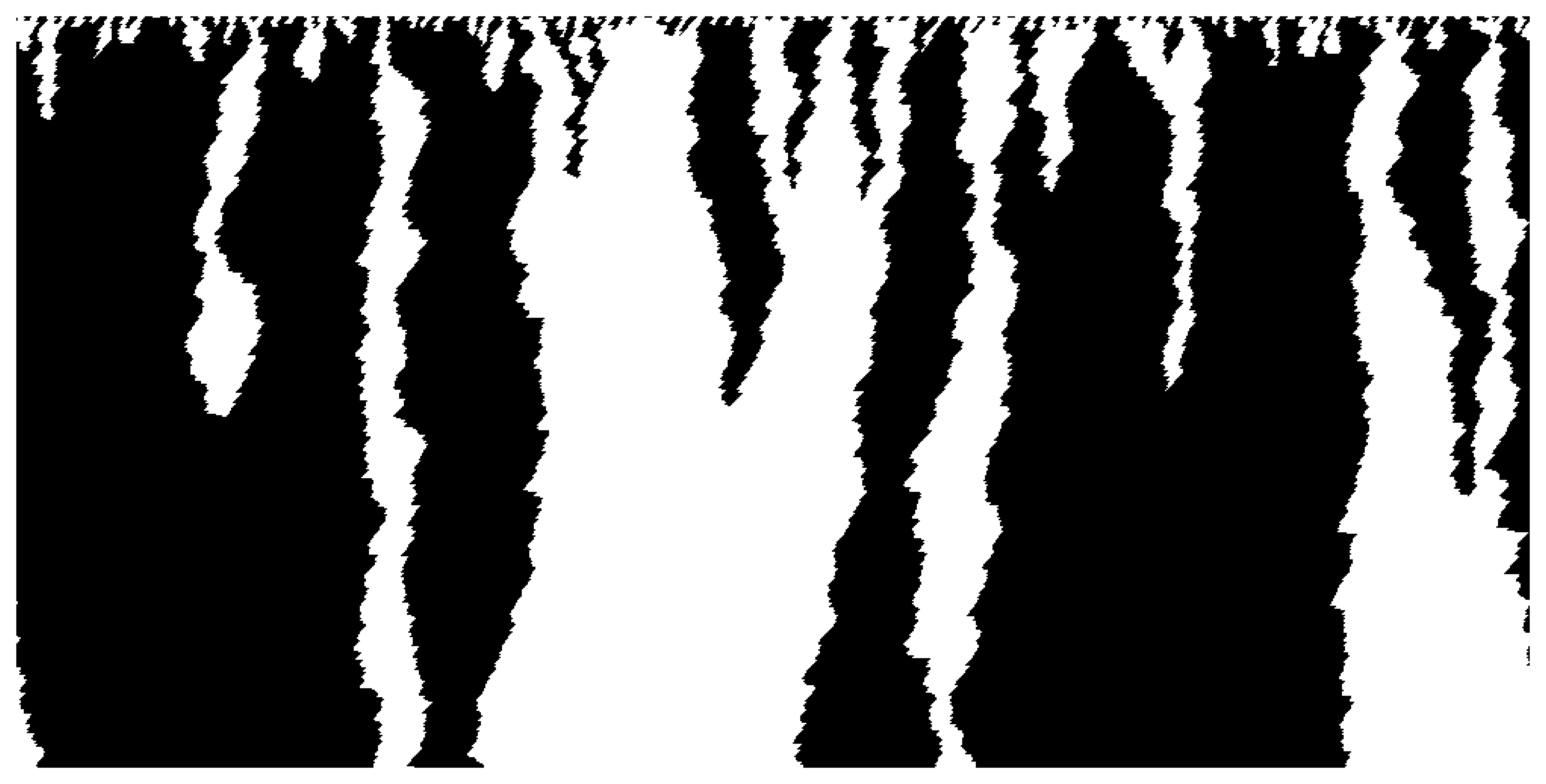}
    }}
    \caption{Space-time plots of long-range and short-range voter models in one dimension.}
    \label{fig: spacetime plot}
\end{figure}
Theorem \ref{thm: main} captures the time-dynamics of the voter model, giving a scaling limit 
in the spatial variable 
$x$ \textit{and the time variable $t$
simultaneously}. More precisely, we show the random walks $(S(\cdot,t))_{t\in\mathbb{N}}$, 
defined by $S(n,t) := \sum_{m=1}^n X(n,t)$, 
rescale to a family of fractional 
Brownian motions $(W_{\alpha}(\cdot,t))_{t\in \mathbb{R}_{\geq 0}}$. 
Just as the voter model is a Markovian evolution of some initial condition, the 
$W_{\alpha} (\cdot,t)$ can be viewed as a Markovian evolution of the initial fraction Brownian $W_{\alpha}(\cdot,0).$ 
This is why we call $W_\alpha$ a \textit{dynamical fractional Brownian motion}.

This viewpoint comes from the work of Hammond and Sheffield in \cite{HS}, who introduced a family of 
simple random walks which rescale to fractional Brownian motion. The walks are reminiscent of an urn scheme 
and can be described intuitively as follows: given the increments $(X_m)_{m<n}$ the $n-$th increment of the walk 
$X_n$ is determined by making an independent sample $k_n$ of a measure $\mu$ on $\mathbb{Z}_{>0}$ and letting $X_n = X_{n-k}$.
This description of the walks is not exactly rigorous, 
since it would take some initial condition to construct them. To get around this 
Hammond and Sheffield use a Gibbs measure formulation. For measures $\mu$ with tails decaying at
a power-law rate $\alpha \in (0,1/2)$,
Hammond and Sheffield show that these walks scale to fractional Brownian motion with
Hurst index $1/2+\alpha$. One reason \cite{HS} is notable is that it gives a simple discrete object that 
scales to fractional Brownian motion. In some sense these walks can be considered the 
natural discrete analogues of fractional Brownian motion. The model leading to our main theorem
was suggested in \cite{HS} as a dynamical variant of the walks in Hammond-Sheffield.

Since the work of Hammond and Sheffield other related models have been considered. 
Bierm{\'e}, Durieu, and Wang generalized the Hammond-Sheffield model to $\Z^d$ in \cite{d2}, and 
these authors explored related urn schemes and scaling limits in \cite{d1}, \cite{d3}, \cite{d4}. 
Recently, in \cite{JL}, Ingelbrink and Wakolbinger gave a new proof of Gaussianity for the 
Hammond-Sheffield model using Stein's method. We also mention the work of Enriquez which predates 
the result of Hammond-Sheffield. Enriquez \cite{NE} showed that one can obtain fractional Brownian motion
by rescaling a sum of many independent random walks, where each individual random walk has some
dependence between steps. 

A difficulty to proving Gaussianity in many of these models is that they exhibit long-range dependence. 
Studying fields with long-range correlations can be challenging as many traditional 
tools rely on strong independence assumptions. One contribution of this work is to adapt the 
Lindeberg swapping trick to this setting for an intuitive proof of Gaussianity. 
Our methods can be adapted to higher dimensions and to give a
quantitative rate of convergence, given quantitative information on $\mu$.

\subsection{Statement of Results}
We define the voter model on $\mathbb{Z}$ with interactions $\mu$ as the Markov chain 
$(\eta_t)_{t\in\mathbb{Z}}$ with state space $\{-1,+1\}^\mathbb{Z}$ and
transition function $P:\{-1,+1\}^\mathbb{Z}\times \{-1,+1\}^\mathbb{Z}\to \mathbb{R}$ given by
\begin{align*}
P(\eta,\cdot) = \rho_\eta(\cdot),
\end{align*}
where $\rho_\eta$ is the product measure on $\{-1,+1\}^\mathbb{Z}$ with marginals 
\begin{align*}
\rho_{\eta}(\{\xi\in \{-1,+1\}^\mathbb{Z} \ |\ \xi(m) = +1\}) = \sum_{n\in\mathbb{Z}} \mu({n-m})1_{\eta(n) = 1}.
\end{align*}

We consider $\mu$ in the following class of measures. 
\begin{definition}
\label{def: measure class}
For any $\alpha\in (0,1)$ let $\Gamma_\alpha$ be the class of probability measures on $\mathbb{Z}$ such that for all $n\in \mathbb{Z}$
\begin{enumerate} 
\item $\mu(\{-n\}) = \mu(\{n\})$,
\item $\mu(\{n,n+1,...\}) = L(n)n^{-\alpha}$, for some slowly varying $L:\mathbb{R}_+\to\mathbb{R}_+$,
\item the greatest common divisor of $\supp(\mu)$ is $1$.
\end{enumerate}
Recall that $L:\R_{+}\to \R_{+}$ is \textit{slowly varying} if for each $c\in (0,\infty)$
$$\lim_{x\to\infty} \frac{L(cx)}{L(x)}= 1.$$
\end{definition}

It is a classical fact that for $\mu\in \Gamma_{\alpha}$ the Markov chain $(\eta_t)_{t\in\mathbb{Z}}$ has a 
a one-parameter family of extremal invariant measure $\{\nu_p\ | p\in [0,1]\}$ which are translation
invariant and characterized by the property $\nu_p(\{\xi\in \{-1,+1\}^\mathbb{Z}\ |\ \xi(0) = 1\}) = p$. 
See Theorem 1.9 in \cite{holllig} or Theorem \ref{thm: extremal measures} below. Note that if $p=0$ or $p=1$, then $\eta_t(n)$ is almost surely 
constant. Our first main theorem is as follows. 

\begin{theorem}\label{cor1}
Let $p,\alpha\in(0,1)$, $\mu\in \Gamma_\alpha$, and $\xi:\Z\to\{+1,-1\}$ have law $\nu_p$. Define 
$F^n_p\in (\mathcal{S}(\mathbb{R}))^{'}$ to be 
the random tempered distribution defined by
$$F^n_p (\phi) := \frac{1}{\sigma_n\sqrt{\alpha(\alpha/2 + 1/2)}}\sum_{i\in\Z}\phi\left(\frac{i}{n}\right)(\xi(i) - (2p-1)),$$
for any $\phi\in \mathcal{S}(\mathbb{R})$, where $\sigma_n$ is the deterministic sequence given in Theorem \ref{thm: main}.
Then 
$$F_p^n(\phi)\underset{d}\to \mathcal{N}(0,\langle \phi, (-\Delta)^{-s}\phi\rangle)$$ 
as $n\to\infty$. In the language of generalized random fields, $F_p^n$ converges in distribution, as a random tempered distribution, to $FGF_{\alpha/2}$, the fractional Gaussian field on 
$\R$ with index $\frac{\alpha}{2}$, defined below. 
\end{theorem}

The fractional Gaussian field on $\R^d$ with index $h$, denoted $FGF_{h}(\R^d)$, 
is defined as a random tempered distribution given by $(-\Delta)^{-h/2} W$, where $W$ is Gaussian white noise on $\R^d$. 
$FGF_{\alpha/2}$ is known as fractional Gaussian noise and can be intepreted as an mean $0$, Gaussian process
indexed by the Schwarz functions with covariance 
\begin{align*}
\Cov(FGF_{\alpha/2}(\phi), FGF_{\alpha/2}(\psi)) = \langle \phi,(-\Delta)^{-\alpha/2} \psi\rangle = \int_{\mathbb{R}}\int_{\mathbb{R}}\frac{\phi(x)\psi(y)}{|x-y|^{1-\alpha}}dx\ dy.
\end{align*}
For more on fractional Gaussian fields the reader is encouraged to consult the survey \cite{lssw}. 

The above theorem is a corollary of our main results. To state our main results, 
we introduce the random directed graph $\mathcal{G}_{\mu}$ with vertex
set $\mathbb{Z}^2$ defined as follows: for each $(n,t)\in \Z^2$ make 
an independent sample of $\mu$, denoted $J_{(n,t)}$, and attach an outward directed edge starting at $(n,t)$ and pointing to $(n-J_{(n,t)},t-1)$. 
\begin{figure}[htp]
    \centering
    \includegraphics[width=12cm, height  = 7cm]{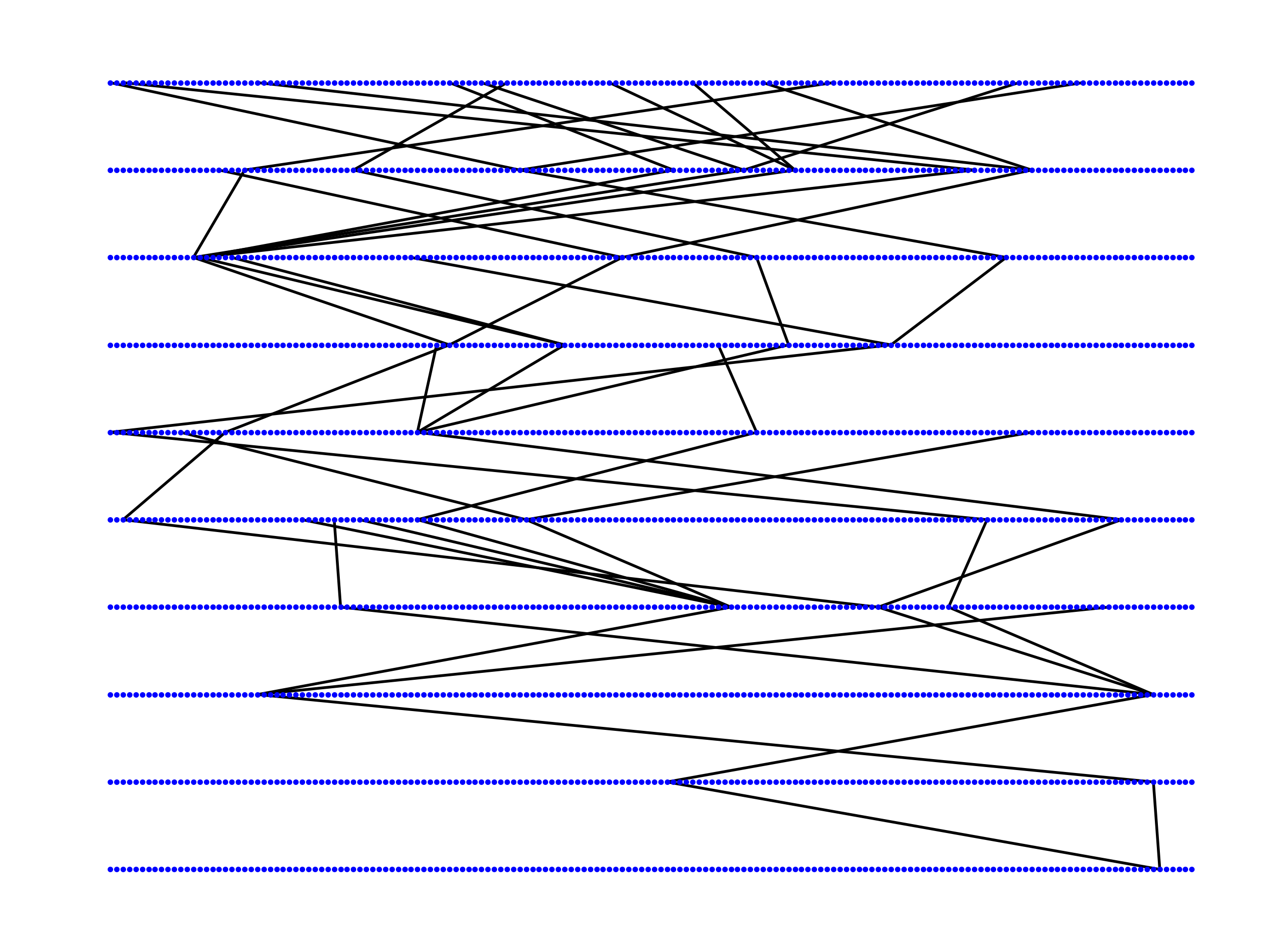}
    \caption{A connected component in a sample of $\mathcal{G}_\mu$, for a truncated $\mu\in\Gamma_{1/2}$ }
    \label{fig:sample}
\end{figure}

For $p\in (0,1)$ define the random field $X_p:\Z^2\to \{+1,-1\}$ 
by sampling the random graph $\mathcal{G}_\mu$ and independently assigning $+1$ or $-1$ to each connected component of $\mathcal{G}_\mu$
with probabilities $p$ and $1-p$ respectively. Precisely, if $(T_\beta)_{\beta \in I}$ denotes 
the (random) connected components of $\mathcal{G}_\mu$, and $(B_\beta)_{\beta\in I}$ is a collection of i.i.d. 
random variables taking the values $+1$ or $-1$, with probabilities $p$ and $1-p$ respectively, then we define 
$X_p(n,t) := \sum_{\beta\in I}1_{T_\beta}(n,t)B_\beta.$ Finally, define the random field  
\begin{equation}
\label{defn: S(n,t)}
S(n,t):= \begin{cases} 
    \sum_{i=1}^n X_p(i,t) & n > 0 \\
    -\sum_{i=0}^{n} X_p(i,t) & n \leq 0.
 \end{cases} 
\end{equation}

\begin{theorem}\label{thm: main}
Let $p,\alpha\in(0,1)$, $\mu\in\Gamma_\alpha$, and define $\sigma_n$ by 
$$\sigma_n^2 := \tilde{c}_pn^{1+\alpha}L(n)^{-1},$$
where $\tilde{c}_p$ is the deterministic constant given in Lemma 3.1. Then the random field 
$$S_n(x,t):\R^2\to \R^2\ :\  (x,t)\to \frac{1}{\sigma_n}\Big(S(\lfloor xn \rfloor,\lfloor tn^{\alpha}L(n)^{-1}\rfloor)-(2p-1)\lfloor xn\rfloor \Big)$$
converges in distribution to the Gaussian random field $W_\alpha(x,t)$, defined below, as $n\to\infty$.
\end{theorem}

\begin{remark}
From the proof it is clear that one could replace the independent random variables $(B_\beta)_{\beta\in I}$ 
used to construct $S(n,t)$ with any family of independent random variables having 3 bounded moments. 
In this sense the proof of Theorem \ref{thm: main} yields an invariance principle. 
\end{remark}

\begin{definition}
For each $\alpha\in(0,1)$, define $W_\alpha(x,t)$ to be the unique mean $0$ Gaussian random field on $\R^2$ with covariance
\begin{multline}\label{covariance function}
\Cov(W_{\alpha}(x_1,t_1), W_{\alpha}(x_2,t_2)) = \\
\frac{1}{2V(0,1)} \Big( V(t_2-t_1,x_1)|x_1|^{1+\alpha} + V(t_2-t_1,x_2)|x_2|^{1+\alpha} - V(t_2-t_1,x_2-x_1)|x_2-x_1|^{1+\alpha}\Big),
\end{multline}
where 
\begin{align*}
V(t,x)  
& := \int_0^{\infty}\frac{1-\cos(u)}  {u^{2+\alpha}}e^{-c_\alpha\frac{t}{|x|^\alpha}u^\alpha}\, du\\
c_\alpha & := \cos(\frac{\alpha \pi}{2})\Gamma(1-\alpha).
\end{align*}
\end{definition}

$W_{\alpha}$ is self-similar, has correlations in the
 time direction which decay exponentially, and can be seen 
as a kind of fractional Ornstein-Uhlenbeck process. Further features of the field $W_\alpha$ are discussed in Section \ref{sect: discussion}. 
To motivate our next theorem we note that at each $t\in \mathbb{R}$ the process $W(\cdot,t)$ is fractional Brownian 
motion with Hurst parameter $1/2+\alpha/2$.
Indeed, recall that fractional Brownian motion with Hurst parameter $H\in (0,1)$, which we denote $B^H$, is 
the unique self-similar Gaussian process on $\mathbb{R}$
with stationary increments, and covariance given by
\begin{align*}
\E(B^H(t)B^H(s)) = \frac{1}{2}(|t|^{2H} + |s|^{2H} - |t-s|^{2H}).
\end{align*}
Note that Brownian motion occurs when $H=1/2$, and $B^H$ exhibits long-range correlations when $H>1/2$ in the sense that 
$$\sum_{n=1}^\infty \mathbb{E}[B^H(1)(B^{H}(n+1)-B^{H}(n))] = \infty.$$
Our second main theorem says that $S_n(\cdot,0)$ converges to $B^H$ in the stronger $L^\infty$ topology.
\begin{theorem}
\label{thm:fbm}
Let $p,\alpha\in(0,1)$ and $\mu\in \Gamma_{\alpha}$. There exists 
a sequence of couplings $C_n$ of $S_n(\cdot,0)$ with $B^{1/2+\alpha/2}$ such that
for all $T,\epsilon>0$ 
$$\lim_{n\to\infty}C_n(||S_n(\cdot,0)-B^{1/2+\alpha/2}||_{L^\infty([0,T])}\geq \epsilon ) = 0. $$ 
\end{theorem}

\begin{remark}
Theorem \ref{cor1} follows readily from the above once we show that $(X_p(n,0))_{n\in\mathbb{Z}}$ 
has law $\nu_p$, and thus $(X_p(n,0))_{n\in\mathbb{Z}}$ is a sample of the voter model in equilibrium. 
This is done in Section 7. 
\end{remark}

\begin{remark}
It is an interesting question if the topology of convergence in Theorem \ref{thm: main} could
be improved to $L^\infty$. 
To see the subtlety of this question compare the increments of $S(n,t)$ in $n$ and $t$. 
On one hand, when $n\geq 0$ we have $|S(n+1,t)-S(n,t)| = |X_p(n+1,t)|= 1$ almost surely.
However $S(n,t+1)-S(n,t) = \sum_{j=1}^n X_p(j,t+1)-X_p(j,t)$ is a sum of $2n$ values of $X_p(\cdot,\cdot)$ 
and will be of much larger size than $1$, typically power law size in fact. 
Improving the topology of convergence in Theorem \ref{thm: main} is difficult 
because $S(n,t)$ has much less regularity in the $t$-direction. 
\end{remark}

\subsection{Discussion of the Scaling Limit $W_{\alpha}$}
\label{sect: discussion}

To our knowledge the field $W_{\alpha}$ has not appeared in the literature before, and it is interesting to consider some aspects of it. 
It is self-similar under the scaling $W_\alpha(x,t) \stackrel{d}{=} \frac{1}{r^{\frac{1}{2}(1+\alpha)}}W(rx,r^\alpha t)$. 
In particular, for each fixed $t\in\R$, $W_\alpha(\cdot,t)$ is fractional Brownian 
motion with Hurst parameter $1/2+\alpha/2$. In this sense, $W_\alpha$ may be viewed as a one-parameter family of evolving fractional
Brownian motions.

In the $\alpha\to 0$ limit, the Gaussian processs $W_\alpha(x,t)$ takes on the mean and covariance of
Brownian motion evolving under the Ornstein-Uhlenbeck process. Heuristically, one can see this in the dynamics of the voter 
model as follows. At time $0$ sample an $n$-step simple random walk. At time $1$ update the random walk by choosing one of the $n$ steps uniformly at 
random and independently resampling the $\pm 1$ value there. Repeating this update process gives a one-parameter family of $n-$step simple random walks indexed by time. 
Rescaling this family of random walks produces a family of Brownian motions evolving according to Ornstein-Uhlenbeck processes. In some sense 
this is the $\alpha\to 0$ limit of the voter model dynamics. Indeed, as $\alpha\to 0$ the measure $\mu$ has larger and larger tails so from time $t$ to $t+1$, the $\pm1$ value at $(n,t+1)$ is being 
sampled randomly from $X_p(\cdot,t)$ over an increasingly large range, making it, effectively, an independent 
random $+1$ or $-1$ value chosen with proabailities $p$ and $1-p$.

We finish by noting that $W_\alpha(x,t)$ is almost surely $\gamma$- holder continuous for all $\gamma <1/2$ in the $t$-direction, and 
is $\gamma$-Holder continuous for all $\gamma< 1/2+\alpha/2$ in the $x$-direction. One can see this by applying the Kolmogorov continuity criterion
to $W_\alpha$, using the definition to compute the necessary moments.

\subsection{A Brief Proof Sketch}
\label{sketch}

The heuristics underlying the proof are that if $T_{(k,t)}$ denotes the connected component of 
$G_\mu$ containing $(k,t)$, then 
with very high probability $|T_{(k,t)}\cap[0,n]\times\{t\}| \sim n^{\alpha}$ for $k\in[0,n]$.
Hence, if $I_{n,t}$ denotes set of connected components of $\mathcal{G}_{\mu}$ which intersect
$[0,n]\times\{t\}$ one should have
\begin{equation}
\label{eq:heuristic}
S(n,t)= \sum_{i=1}^n X_p(i,t) = \sum_{\beta\in I_{n,t}}B_\beta |T_\beta\cap[0,n]\times\{t\}|\approx 
    \sum_{j=1}^{n^{1-\alpha}} B_j n^\alpha,
\end{equation}
where the $B_j$ are independent Bernoulli random variables. By the central limit 
theorem, the expression on the RHS is asymptotically Gaussian with variance $n^{1+\alpha}$. 

The heuristic $|T_{(k,t)}\cap[0,n]\times\{t\}| \sim n^{\alpha}$ comes from noting that 
the probability a random walk with jump distribution $\mu$ started at $0$ ever 
intersects a random walk with jump distribution $\mu$ started at $m$, decays like $m^{-1+\alpha}$, 
and so it is reasonable to expect that 
$$\P((j,t) \in T_{(k,t)}) \sim \langle k-j\rangle ^{-1+\alpha},$$
which gives
$$\mathbb{E}|T_{(k,t)}\cap[0,n]\times\{t\}| =  \sum_{j=0}^n \P((j,t) \in T_{(k,t)}) \sim  n^{\alpha}.$$ 

This paper implements these heuristics rigorously, and is organized as follows. 
In Section 2 we show the random graph $\mathcal{G}_\mu$ has infinitely many connected components when $\alpha\in(0,1)$. 
It is natural to expect that this happens for $\alpha\in (0,1)$ since a random walk with step 
distribution $\mu$ is recurrent whenever $\alpha>1$ and transient 
whenever $\alpha<1$ \cite{Spitz}. The next step is to derive the asymptotic covariance of the field $S(n,t)$. 
This is done with a Fourier analytic argument, using an idea from \cite{HS}. We note that the 
techniques in \cite{HS} need to be revised significantly to deal with the multi-dimensional nature of the model. 

To finish the proof of Theorem \ref{thm: main} we need to prove $S(n,t)$ is asymptotically Gaussian, which is done 
in Section 5. As opposed to most prior works which use delicate martingale central limit theorems, we take a component-based 
approach using the Lindeberg exchange method to implement the heuristic given above. As suggested
in \eqref{eq:heuristic}, if all the $|T_\beta\cap [0,n]\times \{0\}|$ were deterministic and 
of size $n^\alpha$, one could apply the central limit theorem to conclude the sum is nearly Gaussian. 
While the $|T_\beta\cap [0,n]\times \{0\}|$ 
are not deterministic, they are independent of the random variables $(B_\beta)_{\beta\in I}$, and so 
we argue that we may "swap" the $B_\beta$'s for independent unit Gaussians, without changing the limiting distribution
of $S(n,t)$. 
With the $B_\beta$'s replaced by Gaussians we now have a sum like 
$$\sum_{\beta\in I_{n,t}} g_{\beta}|T_{\beta}\cap[0,n]\times \{t\}|,$$
where $(g_{\beta})_{\beta}$ are indepdent unit Gaussians. Conditioned on the $T_{\beta}$, this new sum is a Gaussian with 
variance $\sum_{\beta\in I_{n,t}}|T_{\beta}\times[0,n]\times\{t\}|^2$. Proving Gaussianity of the rescaled $S(n,t)$ now amounts to proving a law of large 
numbers for that quantity. The main technical points here are moment and covariance bounds on $|T_\beta\cap [0,n]\times \{t\}|$ proved in Section $4$.

Once Theorem \ref{thm: main} is proven, Theorem \ref{thm:fbm} follows 
by noting that Lemma \ref{cov} implies that $S_n(x,0)$ is Holder-continuous, and this is done in Section 6. Finally,
to prove Theorem \ref{cor1} the main
thing to check is that a time slice of $X_p(n,t)$ has the law of $\nu_p$. This is done in 
Section 7. 
\subsection{Notation}
Throughout the paper $C > 1 > c > 0$ denote
constants that may differ in each instance and depend on $\mu$. Any further dependencies will be explicitly stated. 
We also use the notation $a_n = b_n + o(c_n)$ if $\frac{|a_n-b_n|}{|c_n|}\to 0$ as $n\to\infty,$
and $a_n  = b_n + O(c_n)$ if for some constant $C>0$ $|a_n-b_n|\leq C|c_n|$ for all $n$ .

\subsection{Acknowledgements}
The author is grateful to Alan Hammond for suggesting this problem and stimulating discussions. The author would also like to thank Charles Smart for useful discussions and encouragement, and Adam Black for helpful comments on an earlier draft of the paper. 

\section{Number of Components of $\mathcal{G}_\mu$}
Our first result relates the rate of tail decay of $\mu$ to the number of connected 
components of $\mathcal{G}_\mu$. Recall that we defined $\mathcal{G}_{\mu}$ as a random directed
graph on $\mathbb{Z}^2$ in which each $(n,t)$ has an edge pointing to $(n-J_{(n,t)}, t-1)$, where
the $(J_{(n,t)})_{(n,t)\in\mathbb{Z}^2}$ 
are i.i.d. samples of $\mu$. We  define the following. 

\begin{definition}
The \textit{parent} of vertex $(n,t)$ is the vertex $(n-J_{(n,t)},t-1)$.
\end{definition}

\begin{definition}
The \textit{ancestral line of $(n,t)$} is the unique sequence of vertices starting at $(n,t)$ 
in which the next vertex in the sequence is the parent of the current vertex
\end{definition}

\begin{definition}
We write $(n_1,t_1)\sim (n_2,t_2)$ if the ancestral lines of 
$(n_1,t_1)$ and $(n_2,t_2)$ merge at some point.
\end{definition}

Note the law of the 
ancestral line of $(n,t)$ is given by a random walk started at $(n,t)$ with jump 
distribution $\mu$, and two points, say $(n_1,t_1)$ and $(n_2,t_2)$, are in the same connected component 
if and only if their ancestral lines merge at some point, i.e. $(n_1,t_1)\sim (n_2,t_2)$. 

\begin{proposition}\label{components}
Suppose $\mu\in\Gamma_\alpha$. If $\alpha>1$ then $\mathcal{G}_\mu$ has one connected component almost surely, and if $\alpha<1$ then $\mathcal{G}_\mu$ has infinitely many connected components almost surely. 
\end{proposition}

\begin{proof}
Let $A_{(0,0)}$ be the ancestral line of $(0,0)$ and let $\Tilde{A}_{(0,0)}$ be an i.i.d. sample of $A_{(0,0)}$.

\textit{Claim 1:} If $\E |A_{(0,0)}\cap\Tilde{A}_{(0,0)}| = \infty$ then $\mathcal{G}_\mu$ has one connected component almost surely. If $\E|A_{(0,0)}\cap\Tilde{A}_{(0,0)}| < \infty$, 
then $\mathcal{G}_\mu$ has infinitely many connected components almost surely. 

Since the greatest common divisor of the support of $\mu$ is $1$, we have that $\E|A_{(0,0)}\cap\Tilde{A}_{(0,0)}| = \infty$ implies $\P((n,0)\sim (0,0))=1$ for each $n\in\Z$. 
So $\mathcal{G}_\mu$ has 1 connected component almost surely. On the other hand if $\E|A_{(0,0)}\cap \Tilde{A}_{(0,0)}|< \infty$ then 
we argue that $\P((k,0)\sim (0,0))\to 0$ as $k\to \infty$. To see this let $q_{n,t} := \P((n,t)\in A_{(0,0)})$ and note that 
$$\E|A_{(0,0)}\cap\Tilde{A}_{(0,0)}| = 1+\sum_{n\in\Z, t<0} q_{n,t}^2<\infty.$$
This implies that if $A_{(k,0)}$ is an ancestral line started at $(k,0)$ and 
\textit{independent} of $A_{(0,0)}$ 
$$\E|A_{(0,0)}\cap A_{(k,0)}|\to 0$$ 
as $k\to\infty$, since 
$$\E|A_{(0,0)}\cap A_{(k,0)}| = \sum_{n\in\Z, t<0} q_{n-k,t}q_{n,t}\to 0\ as\ k\to\infty.$$
But we have the relation
$$\E|A_{(k,0)}\cap A_{(0,0)}| = \P((k,0)\sim (0,0))\E|A_{(0,0)}\cap\Tilde{A}_{(0,0)}|,$$
which implies that $\P((k,0)\sim (0,0))\to 0$ as $k\to \infty$. Now we argue that
if a sequence $a_0<a_1<...\to \infty$ increases rapidly enough, the points $(a_i,0)$ lie in infinitely many distinct
connected components almost surely. This is an application of Borel-Cantelli. Define the events
$$E_l:=\{(a_l,0)\not\sim (a_j,0)\ \forall 0\leq j<l\}.$$
In words, $E_l$ is the event that $(a_l,0)$ does not lie in the same connected
component as any $(a_j,0)$ with $j\in [0,l)$. 
Since $\P((k,0)\sim (0,0))\to 0$ as $k\to \infty$, if the 
sequence $(a_n)$ increases rapidly enough we have that $\P(E_l^c)\leq \sum_{k=0}^l \P((a_l,0)\sim (a_k,0))\leq \frac{10}{l^2}$. 
Now by Borel-Cantelli infinitely many $E_j$ happen almost surely and hence 
$\mathcal{G}_\mu$ has infinitely many connected components almost surely.  

\textit{Claim 2:} $\E|A_{(0,0)}\cap\Tilde{A}_{(0,0)}| = \infty$ if $\alpha > 1$ and $\E|A_{(0,0)}\cap\Tilde{A}_{(0,0)}| < \infty$ if  $\alpha <1$. 

Let $p_n := \mu (\{n\})$ and define the functions $P:\mathbb{R}/2\pi\mathbb{Z}\to \mathbb{C}$,
$Q:(\mathbb{R}/2\pi\mathbb{Z})^2\to\mathbb{C}$ by 
\begin{align}
\label{def: P and Q}
P(x)
& := \sum_{n\in\mathbb{Z}}e^{inx}\\
Q(x,y) & := \frac{1}{1-e^{-iy}P(x)}.
\end{align}
Note that $\widehat{Q}(n,t) = q_{n,t} = \P((n,t)\in A_{(0,0)})$ for all $(n,t)\in \mathbb{Z}\times\mathbb{Z}_{< 0}$, and thus 
$$\E|A_{(0,0)}\cap\Tilde{A}_{(0,0)}| = 1+ \sum_{(n,t)\in\mathbb{Z}\times\mathbb{Z}_{<0}} q_{n,t}^2 = ||Q||^2_{L^2(\mathbb{T}^2)}.$$
This allows us to relate $\mathbb{E}|A_{(0,0)}\cap\Tilde{A}_{(0,0)}|$ to 
the behavior of $P(x)$, which is just the characteristic function of $\mu$, near $0$. 
Since $\mu\in \Gamma_\alpha$ we have by Theorem 1 in \cite{GL}
\begin{equation}\label{nearzero}
|1-P(x)| = c_\alpha x^\alpha L(x^{-1}) + o(x^{\alpha}L(x^{-1}))
\end{equation}
as $x\to 0$, where 
\begin{equation}
\label{eq: calpha}
    c_\alpha := \cos\left(\frac{\alpha\pi}{2}\right) \Gamma(1-\alpha).
\end{equation}
Using these asympototics we see that $\int_{\mathbb{T}^2}\frac{1}{|1-e^{-iy}P(x)|^2}\, dx\, dy = \infty$ if $\alpha >1$, and is finite if $\alpha <1$. 
\end{proof}
We note that when $\alpha =1$ the number of connected components of $\mathcal{G}_\mu$ can depend on the slowly varying function $L(x)$. 

\section{Asymptotics for the Variance}

For the remainder of the paper we assume $\mu\in\Gamma_\alpha$.

\begin{lemma}\label{variances} For any $p\in(0,1)$ and $t,x_1,x_2\in\R$ we have 
\begin{equation}
\label{covar}
\begin{aligned}
\Cov\Big(S
& (\lfloor x_1n\rfloor,0),\ S(\lfloor x_2n\rfloor ,\lfloor tL(n)^{-1}n^{\alpha}\rfloor)\Big) \\
& = \tilde{c}_p n^{1+\alpha}L(n)^{-1}\Big(V(t,x_1)|x_1|^{1+\alpha}+V(t,x_2)|x_2|^{1+\alpha} - V(t,|x_2-x_1|)|x_1-x_2|^{1+\alpha}\Big) 
+ o(n^{1+\alpha}L(n)^{-1}),
\end{aligned}
\end{equation}
where 
\begin{align*}
V(t,x)  
& := \int_0^{\infty}\frac{1-\cos(u)}  {u^{2+\alpha}}e^{-c_\alpha\frac{t}{x^\alpha}u^\alpha}\, du \\
\tilde{c}_p & := \frac{2p(1-p)}{\pi c_\alpha ||Q||_{L^2}^2},
\end{align*}
and the constant $c_\alpha$ is given in $\eqref{eq: calpha}$. 
\end{lemma}

\begin{remark}
Lemma 3.1 says that the correct scale of $t$ to study the fluctuations of $S(n,t)-S(n,0)$
is $t\sim n^\alpha$. This may be surprising at first glance. A heuristic explanation for 
this timescale is that a random 
walk with step distribution $\mu$ will take $\sim n^\alpha$ steps to leave a 
neighborhood of size $n$. So if $t<<n^\alpha$, the ancestral line of 
almost all any point in $[0,n]\times\{t\}$ will hit some point in $[0,n]\times\{0\}$,
and so for $t<<n^\alpha$, $S(n,t)$ and $S(n,0)$ with me strongly linked. 
Lemmas \ref{heatkernel} and \ref{dissipative} make this heuristic more precise.
\end{remark} 
We will make use of the following classical fact about slowly varying functions. See Theorem 1.5.6 in \cite{BGT} for a proof.

\begin{theorem} [Potter's Theorem]
\label{thm: Potter's}
If $L:(0,\infty)\to(0,\infty)$ is a slowly varying function, 
as defined in definition \ref{def: measure class}, then for each $\epsilon>0$ there exists 
an $x_0>0$ such that for all $x,y>x_0$
$$\frac{L(x)}{L(y)}\leq 2 \max \{(\frac{x}{y})^\epsilon, (\frac{y}{x})^\epsilon\}.$$
\end{theorem}

\begin{proof}[Proof of Lemma \ref{covar}]
We write the covariance above as as a weighted sum of the probabilities $\P((k,t)\sim (0,0))$, and
then deduce the asymptotics of this weighted sum from asynmptotics of the characteristic function of $\mu$. 

By symmetry and translation invariance of the law of $X_p$, it suffices to consider $x_1,x_2,t\in \R_{\geq 0}$. Let $n_1= \lfloor x_1n\rfloor,\ n_2 = \lfloor x_2n\rfloor$, $t_1 = \lfloor tL(n)^{-1}n^{\alpha}\rfloor$ and compute
\begin{align*}
\Cov\Big(S(n_1,0),\ S(n_2,t_1)\Big)
& = \sum_{0\leq k_2\leq n_2}\sum_{0\leq k_1\leq n_1}\Cov(X_p(k_1,0),\ X_p(k_2,t_1))\\
& = 4p(1-p) \sum_{0\leq k_2\leq n_2}\sum_{0\leq k_1\leq n_1} \P((k_1-k_2, t_1)\sim(0,0)) .
\end{align*}
Counting the number of pairs $(k_1,k_2)\in [0,n_1]\times [0,n_2]$ such that $|k_1-k_2|=k$ we see that 
\begin{multline}\label{fouriersum}
\Cov\Big(S(n_1,0),\ S(n_2,t_1)\Big) = 4p(1-p)\Bigg(\sum_{k=0}^{n_1} (n_1-k)\P((k,t_1)\sim(0,0)) + \sum_{k=0}^{n_2} (n_2-k)\P((k,t_1)\sim(0,0)) \\
- \sum_{k=0}^{|n_2-n_1|} (|n_2-n_1|-k)\P((k,t_1)\sim(0,0)) \Bigg)+ O(n).
\end{multline}
It suffices to work out the asymptotics for $\sum_{k=0}^{n_1}(n_1-k)\P((k,t_1)\sim (0,0))$ since the other sums are of the same form. The key observation is that if $\Tilde{A}_{(0,0)}$, $A_{(0,0)}$, and $A_{(k,t)}$ are all independent ancestral lines started at $(0,0),\ (0,0),\ $ and $(k,t)$, respectively, then 
$$\frac{\E |A_{(0,0)}\cap A_{(k,t)}|}{\E |A_{(0,0)}\cap \Tilde{A}_{(0,0)}|} = \P((0,0)\sim (k,t)).$$ 
The above relation allows us to relate the asymptotics of the above sums to integral expressions involving the characteristic function of $\mu$. 
Recall the characteristic function $P:\mathbb{R}/2\pi\mathbb{Z}\to \mathbb{C}$, of $\mu$, given by 
\begin{equation*} 
    P(x) = \sum_{n\in\Z}\mu(\{n\}) e^{inx}.
\end{equation*}
We have that for each $t\in \mathbb{Z}_{\geq 0}$
$$\E |A_{(0,0)}\cap A_{(k,t)}| = \sum_{s=0}^{\infty} \widehat{P^{t+2s}}(k) = \widehat{\frac{P^{t}}{1-P^2}}(k),$$ 
and so we can compute for any $t_1\geq 0$
\begin{align*}
\sum_{k=0}^{n_1} (n_1-k)\P((k,t_1)\sim (0,0)) &= \frac{1}{\E |A_{(0,0)}\cap \Tilde{A}_{(0,0)}|}\sum_{k=0}^{n_1} (n_1-k)\E |A_{(0,0)}\cap A_{(k,t_1)}| \\
&   = \frac{1}{||Q||^2_{L^2}}\sum_{k=0}^{n_1} (n_1-k)\widehat{\frac{P^{t_1}}{1-P^2}}(k)\\
&   = \frac{1}{2\pi ||Q||^2_{L^2}}\int_0^{2\pi} \left(\sum_{k=0}^{n_1} (n_1-k)\cos(kx)\right)\frac{P^{t_1}(x)}{1-P^2(x)}dx.
\end{align*}
The function $Q(x,y)$ was defined in \eqref{def: P and Q}. The sum inside the integral is $\frac{n_1}{2} F_{n_1}(x)+\frac{n_1}{2}$, where $F_N(x) := \sum_{|k|\leq N} \big(1-\frac{|k|}{N}\big)e^{ikx} = \frac{1}{N}\big(\frac{1-\cos(Nx)}{1-\cos(x)}\big)$ 
is the well known Fejér kernel. Thus
\begin{equation}\label{integralexpression}
\begin{split}
\sum_{k=0}^{n_1} (n_1-k)\P((k,t_1)\sim (0,0)) 
& = \frac{n_1}{2\pi ||Q||^2_{L^2}}\int_0^\pi \frac{1-\cos(n_1x)}{n_1(1-\cos(x))}\frac{P^{t_1}(x)}{1-P^2(x)}dx + O(n_1)\\
& = \frac{n_1^{1+\alpha}L(n_1)^{-1}}{2\pi ||Q||_2^2}\int_{0}^{\pi n_1}\frac{1-\cos(u)}{n_1^2\big(1-\cos\big(\frac{u}{n_1}\big)\big)}P^{t_1}\Big(\frac{u}{n_1}\Big)\frac{L(n_1)}{n_1^{\alpha}\big(1-P^2\big(\frac{u}{n_1}\big)\big)}du + O(n_1)\\
& := \frac{n_1^{1+\alpha}L(n_1)^{-1}}{2\pi ||Q||_2^2} \int_{0}^{\pi n_1} A_n(u)B_n(u)C_n(u) + O(n_1),
\end{split}
\end{equation}
using that the integrand was even to integrate from $0$ to $\pi$ and making the change of variables $u=n_1x$ in the second step. Recall that $(n_1,t_1) = (\lfloor x_1n\rfloor , \lfloor tL(n)^{-1}n^{\alpha}\rfloor)$. 

To compute the asymptotics in $n$ we claim we may apply dominated convergence. The standard bound on the Fejér kernel implies $|F_N(x)|\leq C\min(N,\frac{1}{Nx^2})$ (see \cite{Schlag} chapter 1) 
and by definition $|P(x)|\leq 1$ so that $|B_n(u)A_{n}(u)|\leq \frac{C}{1+u^2}$. 
Further by (\ref{nearzero}) and Potter's Theorem, Theorem \ref{thm: Potter's} above, applied to $C_n(u)$ we obtain
$$|A_n(u)B_n(u)C_n(u)|\leq \Big(\frac{C}{u^{\alpha}(1+u^2)}\Big)\frac{L(n_1)}{L(\frac{n_1}{u})}\leq 
\frac{C(u^{\epsilon}+u^{-\epsilon})}{u^{\alpha}(1+u^2)}$$
for any $\epsilon>0$ where $C$ may depend on $\epsilon$, $x_1$ and $t$. The above function is clearly integrable for $\epsilon$ small, so it suffices to compute the pointwise limits of $A_n(u), B_n(u),$ and $C_n(u)$. Fix $u\in\R$ and estimate
$$A_n(u) =  \frac{1-\cos(u)}{n_1^2\big(1-\cos\big(\frac{u}{n_1}\big)\big)} = \frac{2(1-\cos(u))}{u^2} + O(n_1^{-2}).$$
For $B_n$ we use the asymptotics (\ref{nearzero}) and the fact that $(1-x)^{x}\to e^{-1}$ as $x\to 0$ to estimate
\begin{align*}
B_n(u) 
& = \Big(1- \Big(1-P\Big(\frac{u}{n_1}\Big)\Big))^{t_1}\\
& = e^{-c_\alpha\frac{t}{x^\alpha}u^{\alpha}} + o(1),
\end{align*}
where $c_\alpha$ is as used in (\ref{nearzero}). Lastly, for $C_n$ we again use the asymptotics (\ref{nearzero}) to get
\begin{align*}
C_n(u) 
& = \frac{L(n_1)}{n_1^\alpha(1-P^2(\frac{u}{n_1}))}\\
&  = \Big(\frac{1}{1+P(\frac{u}{n_1})}\Big)\frac{L(n_1)}{n_1^\alpha\big(1-P\big(\frac{u}{n_1}\big)\big)}\\
& = \frac{1}{2c_\alpha u^{\alpha}} + o(1)
\end{align*}
Plugging this into (\ref{integralexpression}) gives
$$\sum_{k=0}^{n_1} (n_1-k)\P((k,t_1)\sim (0,0)) = \frac{n_1^{1+\alpha}L(n_1)^{-1}}{2c_\alpha \pi ||Q||_2^2}\int_{0}^{\infty} \frac{1-\cos(u)}{u^{2+\alpha}} e^{-\frac{c_\alpha t}{x^\alpha}u^\alpha}du + O(n_1).$$
Computing the asymptotics of the other terms in (\ref{fouriersum}) and recalling that $L(x)$ is slowly varying yields the claim.
\end{proof}

\section{Estimates on the Random Graph $\mathcal{G}_\mu$}

To apply a Lindeberg swapping type argument, the main technical obstacle is to ensure that the 
connected components are roughly the same size. This is achieved by the second moment 
and covariance estimates in this section, which rest on the heat kernel bounds proved below. 

\subsection{ A Heat Kernel Estimate}
We will need the following estimate for random walks with jump kernel $\mu$. 
In some sense, it quantifies the rate of transience of the random walk generated by $\mu$. 

\begin{lemma}\label{heatkernel}
     Let $p_t(x,y)$ be the $t$-step transition kernel of the random walk on 
     $\Z$ generated by $\mu$ and started at $x$. Then for each $\epsilon >0$, 
     there exists $C>0$ depending on $\epsilon$, such that
\begin{equation}\label{infbound}
    ||p_t(x,y)||_\infty \leq Ct^{-1/(\alpha+\eps)}.
\end{equation}
\end{lemma}

Our proof adapts the proof of the local limit law for stable laws given in \cite{Feller}.
We remark that we were unable to find a proof of Lemma \ref{heatkernel} 
in the literature, but under \textit{stronger} assumptions on $\mu$ much more is known to be true. 
For instance if $\mu$ satisfies the pointwise bounds
$$\frac{\kappa^{-1}}{(1+|n|)^{1+\alpha}}\leq p_n = \mu(\{n\}) \leq \frac{\kappa}{(1+|n|)^{1+\alpha}},$$
then it is known (see \cite{Levin} for instance) that
\begin{equation}
    c_2\min \left(t^{-1/\alpha}, \frac{n}{|n|^{1+\alpha}}\right)\leq p_t(0,n) = q_{(n,-t)} \leq c_1\min\left(t^{-1/\alpha}, \frac{n}{|n|^{1+\alpha}}\right).
\end{equation}
Alternatively, if $\widehat{\mu}$ admits an expansion near $0$ of the form 
$$\widehat{\mu}(\xi) = 1-a|\xi|^\alpha + ib\sgn(\xi)|\xi|^\alpha + O(|\xi|^{\alpha +\delta}),$$
for some $a,b\in\mathbb{R}$ and $\delta>0$ then one can prove similar results. See \cite{Jara}. 
However, a general measure in $\Gamma_\alpha$ may not satisfy
either of these conditions. 

\begin{proof}
Fix $\epsilon>0$ sufficiently small, and recall the characteristic function of $\mu$
    $$P(x) = \hat{\mu}(x) = \sum_{m\in\Z}e^{i mx}\mu(\{m\}).$$
By (2.3) combined with Theorem \ref{thm: Potter's} (Potter's Theorem), we have
the following estimate on $P(x)$
\begin{equation}
    |P(x)|\leq 1-c|x|^{\alpha+\epsilon}
\end{equation}
for each $x\in [-\pi,\pi]$. The constant $c$ may depend on $\epsilon$ and $\mu$. Writing $p_t(0,\cdot)$ in terms of $P$ by Fourier inversion and using the above bound gives
\begin{align*}
    |p_t(0,n)|  
    & =  |\frac{1}{2\pi} \int_{0}^{2\pi} e^{-inx}P^t(x)\, dx|\\
    & \leq \frac{1}{2\pi}\int_{|x|\leq t^{-1/(\alpha+\epsilon)}} (1-c|x|^{\alpha+\epsilon})^t \ dx + \frac{1}{2\pi}\int_{t^{-1/(\alpha+\epsilon)}< |x|\leq \pi} (1-c|x|^{\alpha+\epsilon})^t\, dx\\
    & \leq  Ct^{-1/(\alpha+ \epsilon)} + C\int_{t^{-1/(\alpha+\epsilon)}< |x|\leq \pi} (1-c|x|^{\alpha+\epsilon})^t\, dx.
\end{align*}
The second integral is estimated with the change of variables $x=yt^{-1/(\alpha+\epsilon)}$ to get
\begin{align*}
    \int_{t^{-1/(\alpha+\epsilon)}< |x|\leq \pi} (1-|x|^{\alpha + \eps})^t\, dx 
    & = 2t^{-1/(\alpha + \eps)}\int_{1}^{\pi t^{1/(\alpha + \epsilon)}}\Big(1-c\frac{y^{\alpha+\epsilon}}{t}\Big)^t\, dy\\
    & \leq Ct^{-1/(\alpha + \eps)}\int_{1}^{\infty} e^{-cy^{\alpha + \epsilon}}\, dy \\
    & \leq Ct^{-1/(\alpha + \epsilon)}.
\end{align*}
\end{proof}
Using Lemma \ref{heatkernel} we prove the following estimate for the expected amount of time the random walk 
started at $k\in\Z$ and generated by $\mu$ spends in any interval of length $n$.
\begin{lemma}\label{dissipative}
Fix $\epsilon >0$. Then there exists $C>0$, depending on $\epsilon$, such that for any $k\in \Z$ and $n\in \Z_{> 0}$ we have
\begin{align}\label{twopoint}
    \sum_{0\leq n_1\leq n} \sum_{t\geq 0} p_t(0,n_1-k)\leq Cn^{\alpha+\epsilon}.
\end{align}
\end{lemma}

\begin{proof}
Without loss of generality we may take $0<\epsilon<1-\alpha$. By Lemma \ref{heatkernel} we have for each $t\geq 1$
$$\sum_{0\leq n_1\leq n} p_t(0,n_1-k)\leq C n t^{-1/(\epsilon+\alpha)}.$$
We also have the bound 
$$\sum_{0\leq n_1\leq n} p_t(0,n_1-k)\leq 1,$$
for all $t\geq 0$.
Using both we estimate, using that $\epsilon<1-\alpha$
\begin{align*}
    \sum_{t\geq 0} \sum_{0\leq n_1\leq n}p_t(0,n_1-k)
    & \leq \sum_{0\leq t\leq n^{\alpha+\epsilon}}\sum_{0\leq n_1\leq n}p_t(0,n_1-k) + \sum_{n^{\alpha+\epsilon}\leq t}\sum_{0\leq n_1\leq n}p_t(0,n_1-k)\\
    & \leq n^{\alpha + \epsilon} + C n\sum_{n^{\alpha+\epsilon}\leq t}t^{-1/(\alpha + \epsilon)}\\
    & \leq Cn^{\alpha + \epsilon}
\end{align*}
\end{proof}

\subsection{Moment Bounds on Connected Components} 
For any $S\subset \Z^2$ and $q\in\Z^2$ let $T_q^S = \{q_0\in S| q\sim q_0\}$. $T_q^S$ is the set of points in $S$ which are in the same connected component of $\mathcal{G}_\mu$ as $q$. In this section we establish the following two estimates.
\begin{lemma}\label{momentbounds}
Fix $\epsilon>0$ and $S= [0,n]\times \{0\}$. Then there exists $C>0$, depending on $\epsilon$, such that for each $q\in S$ and $n\in \Z_{>0}$
\begin{align}
\E |T_q^S|^2 \leq Cn^{2\alpha + \epsilon}.
\end{align}
\end{lemma}

\begin{lemma}\label{cov}
Fix $\epsilon>0$, and $S = \cup_{j=1}^{k} [0,x_k n]\times\{t_k\}$. Then there exists $C>0$, depending on $x_i$, $t_i$, and $\epsilon$, such that for each $q\in S$ and $n\in \Z_{>0}$
\begin{equation}
    \sum_{q_0\in S} \Cov \Big(|T_q^S|, |T_{q_0}^S|\Big) \leq C n^{3\alpha+\epsilon}.
\end{equation}
\end{lemma}

Lemma \ref{momentbounds} tells us that $|T_{(0,0)}^S| = O(n^{\alpha+\epsilon})$ with high probability, which 
aligns with the heuristic laid out in Section \ref{sketch}. Lemma \ref{cov} tells us that the sizes of distinct 
connected components are weakly correlated. It is instructive to compare the bound in Lemma \ref{cov} to the 
naive bound one obtains by applying Lemma \ref{momentbounds} as follows
$$ \sum_{q_0\in S}  
\sum_{q_0\in S} Cn^{2\alpha + \epsilon} \Cov \Big(|T_q^S|, |T_{q_0}^S|\Big)\leq C|S|n^{2\alpha + \epsilon} = Cn^{1+2\alpha + \epsilon}.$$
The bound in Lemma \ref{cov} improves this bound by a factor of $n^{\alpha -1}$.

\begin{proof}[Proof of Lemma \ref{momentbounds}]
Without loss of generality we can assume $q=(0,0)$. First, we estimate the number of
points in $[0,n]\times \{0\}$ whose ancestral lines merge with that of a given point. 
In what follows we use the shorthand $p_t(n) := p_t(0,n) = p_t(k,k+n)$, 
and define $p_t(k) = 0$ for all $t<0$. Recalling that
$$\P((n_1,0)\sim (j,s)) = \frac{\E|A_{(n_1,0)}\cap A_{(j,s)}|}{\E|A_{(0,0)}\cap\Tilde{A}_{(0,0)}|} = C\sum_{(k,t)\in\Z^2}p_{t}(k-n_1)p_{t+s}(j-k),$$
where $C$ depends on $\alpha$, we have for any $(j,s)\in \mathbb{Z}^2$
\begin{equation}
\begin{aligned}
\label{twopoint probability}
    \sum_{0\leq n_1\leq n}\P((n_1,0)\sim (j,s)) 
    & \leq C  \sum_{0\leq n_1\leq n}\sum_{(k,t)\in\Z^2}p_{t}(k-n_1)p_{t+s}(j-k)\\
    & \leq C \sum_{0\leq n_1\leq n}\sum_{t\geq 0}p_{2t+s}(j-n_1) \\
    & \leq Cn^{\alpha + \epsilon}.
\end{aligned}
\end{equation}
We used the semi-group property of $p_t(x,y)$ summing over $k$ in the second inequality, and then used Lemma \ref{dissipative} to finish. 
To prove Lemma \ref{momentbounds} we compute
\begin{align*}
    \E|T_{(0,0)}^S|^2 = \E\sum_{0\leq n_1,n_2\leq n}1_{(0,0)\sim (n_1,0)}1_{(0,0)\sim (n_2,0)} = \sum_{0\leq n_1,n_2\leq n}\P((0,0)\sim (n_1,0)\sim (n_2,0)).
\end{align*}
The last probability can be bounded as
\begin{equation}
\begin{aligned}
\label{tripsbound}
\P((0,0)\sim (n_1,0)\sim (n_2,0)) 
& \leq C\Big(\sum_{(k,t)\in\Z^2} p_t(k)p_t(n_1-k)\P((k,t)\sim (n_2,0)) \\
& \quad \qquad+ \sum_{(k,t)\in\Z^2}p_t(k)p_t(n_2-k)\P((k,t)\sim (n_1,0))\\
& \quad \qquad + \sum_{(k,t)\in\Z^2}p_t(n_2-k)p_t(n_1-k)\P((k,t)\sim (0,0))\Big )\\
& := C(E_1(n_1,n_2) + E_2(n_1,n_2) + E_3(n_1,n_2)). 
\end{aligned}
\end{equation}
$E_1(n_1,n_2)$ upper bounds the probability that $A_{(0,0)}$ and $A_{(n_1,0)}$ meet first at the point $(k,t)$ and $(n_2,0)$ is in the same connected component as $(k,t)$. 
$E_2(n_1,n_2)$ and $E_3(n_1,n_2)$ are the same but for the cases that $(0,0)$ and $(n_2,0)$ meet first or $(n_2,0)$ and $(n_1,0)$ meet first respectively. 
Consider the first term above. Summing it over $n_1$ and $n_2$ gives
\begin{align*}
\sum_{0\leq n_1,n_2\leq n}E_1(n_1,n_2)
& = \sum_{0\leq n_1,n_2\leq n}\sum_{(k,t)\in\Z^2} p_t(k)p_t(n_1-k)\P((k,t)\sim (n_2,0)) \\
& = C \sum_{0\leq n_1\leq n}\sum_{(k,t)\in\Z^2} p_t(k)p_t(n_1-k)\Big(\sum_{0\leq n_2\leq n}\P((k,t)\sim (n_2,0))\Big)\\
& \leq Cn^{\alpha+\epsilon}\sum_{0\leq n_1\leq n}\sum_{(k,t)\in\Z^2} p_t(k)p_t(n_1-k)\\
& = Cn^{\alpha+\epsilon} \sum_{0\leq n_1\leq n}\sum_{t\geq 0}p_t(n_1) \\
& \leq Cn^{2\alpha +2\epsilon}.
\end{align*}
We summed over $n_2$ and used \ref{twopoint probability} for first inequality, 
applied the semigroup property and then used Lemma \ref{twopoint} to finish. 
Since $\mu$ is symmetric and we are summing
over $n_1,n_2\in[0,n]$, one has the same bound on the contribution of $E_2$ and $E_3$. 
Thus $\E|T_0^S|^2 \leq C n^{2\alpha+2\epsilon}$.
\end{proof}

Next we establish Lemma \ref{cov}. We need the following lemma, which uses a coupling argument inspired by Lemma 7.1 in \cite{JL}.

\begin{lemma}
\label{lem: quartet coalescence}
For any $q_0,q_1,q_2\in \Z^2$
\begin{equation}\label{quartet}
    \Cov(1_{0\sim q_1}, 1_{q_0\sim q_2}) \leq \P(0\sim q_1\sim q_0\sim q_2).
\end{equation}
\end{lemma}

\begin{proof}
We have 
$$\Cov(1_{0\sim q_1}, 1_{q_0\sim q_2}) = \P(0\sim q_1 \cap q_0 \sim q_2)-\P(0\sim q_1)\P(q_0\sim q_2),$$
and we estimate the RHS with a coupling argument. Sample independent ancestral lines $A_0, A_{q_0}, A_{q_1}, A_{q_2}$ which start at $0, q_0, q_1, q_2$, and build a random graph $G$ inductively as follows. Start with $G= A_0$. If $A_{q_0}\cap G \neq \emptyset$ then add to $G$ all points in $A_{q_0}$ \textit{before} the intersection with $G$, otherwise add to $G$ \textit{all} of $A_{q_0}$. Repeat this process with the updated $G$ for $A_{q_1}$ and $A_{q_2}$. Note that the random graph $G$ is a sample of the random graph $\mathcal{G}_\mu$ restricted to the ancestral lines of $0, q_0, q_1, q_2$ and we can estimate
\begin{align*}
    \P(0\sim_G q_1\not\sim_G q_0\sim_G q_2) 
        & \leq \P(\{A_0\cap A_{q_1}\neq \emptyset\}\cap \{A_{q_0}\cap A_{q_2}\neq\emptyset\})\\
        & = \P(\{A_0\cap A_{q_1}\neq \emptyset\})\P(\{A_{q_0}\cap A_{q_2}\neq\emptyset\})\\
        & = \P(0\sim q_1)\P(q_0\sim q_2),
    \end{align*}
    where $x\sim_Gy$ means that $x$ and $y$ are in the same connected component of $G$. (\ref{quartet}) now follows since
    $$\P(0\sim q_1\cap q_0\sim q_2) = \P(0\sim q_1\sim q_0\sim q_2) + \P(0\sim q_1\not\sim q_0\sim q_2).$$
\end{proof}

Now we prove Lemma \ref{cov}.
\begin{proof}[Proof of Lemma \ref{cov}]
We assume without loss of generality that $q=(0,0)$ and we write $0$ for the point $(0,0)$. 
Applying Lemma \ref{lem: quartet coalescence} gives 
\begin{equation}
\begin{split}
\Cov(|T_0^S|, |T_{q_0}^S|) 
& =\sum_{q_1,q_2\in S} \Cov(1_{0\sim q_1}, 1_{q_0\sim q_2})\\
& \leq \sum_{q_0,q_1,q_2\in S} \P(0\sim q_0\sim q_1\sim q_2).
\end{split}
\end{equation}

We can now handle the RHS in the same way as the previous lemma. 
For clarity we give the argument when $k=1$ and $S = [0,n]\times \{0\}$. 
Note that we only use the semi-group property and the bound (\ref{twopoint}), 
hence the same argument works for $S = \cup_{j=1}^k [0,x_j n]\times\{t_j\}$ with straightforward adjustments. 

To estimate the above sum we break the event that $(0,0)\sim(i,0)\sim (k,0)\sim (j,0)$ into cases 
and estimate each case, similar to Lemma \ref{momentbounds}.
There are 18 cases since there are $6$ ways to choose which pair of ancestral lines meet first, 
and then $3$ ways to choose which two ancestral lines coalesce next. Note that these cases are not disjoint, 
since more than two ancestral lines may merge at a given point. However, up to permuting the vertices, 
there are only 2 distinct types of sums corresponding to the two different ways the ancestral lines can coalescences . 
Figure 2 depicts the two possibilities. It suffices to estimate the contributions from those two types. 

\begin{figure}
    \centering
    \subfloat[\centering ]{{
    \includegraphics[width = 6cm, height=5cm]{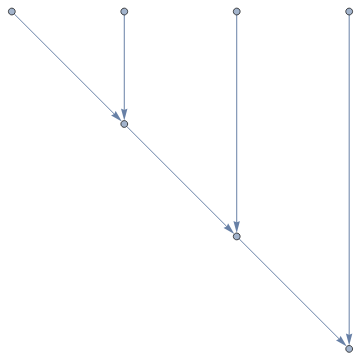}
    }}
    \hspace{3cm}
    \subfloat[\centering ]{{
    \includegraphics[width = 6cm, height=5cm]{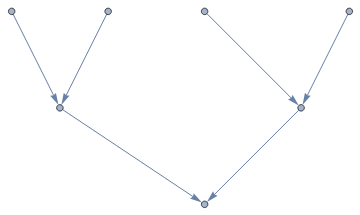}
    }}
    \caption{In (A), the left the ancestral lines coalesce one at a time. In (B) the ancestral lines coalesce pairwise.}
\end{figure}

The first type, depicted in (A) of figure 3, is when the lines coalesce one at a time. Without loss of generality we estimate this by 
the probability that the ancestral line of $(0,0)$ merges with the ancestral line of $(i,0)$ at $(a_1,t_1)$, 
then merges with the ancestral line of $(j,0)$ at $(a_2,t_2)$ (with $t_2\leq t_1$), and finally with that of $(k,0)$. We can upper bound this probability
\begin{align*}
C\sum_{(a_1,t_1), (a_2,t_2)\in \Z^2} p_{t_1}(a_1)p_{t_1}(a_1-i)p_{t_2-t_1}(a_2-a_1)p_{t_2}(j-a_2)\P((a_2,t_2)\sim (k,0)).
\end{align*}
We want to sum over $0\leq i,j,k \leq n$. Summing first in $k$, and then repeatedly using the semi-group property and the estimate (\ref{twopoint}) we get
\begin{align*}
C\sum_{0=i,j,k}^n & \sum_{(a_1,t_1), (a_2,t_2)\in\Z^2} p_{t_1}(a_1)p_{t_1}(a_1-i)p_{t_2-t_1}(a_2-a_1)p_{t_2}(j-a_2)\P((a_2,t_2)\sim (k,0))\\
& \leq Cn^{\alpha+\epsilon}\sum_{0=i,j}^n \sum_{(a_1,t_1), (a_2,t_2)\in\Z^2} p_{t_1}(a_1)p_{t_1}(a_1-i)p_{t_2-t_1}(a_2-a_1)p_{t_2}(j-a_2)\\
& \leq Cn^{\alpha+\epsilon}\sum_{0=i,j}^n \sum_{(a_1,t_1)\in\Z^2} \sum_{t_2\in\Z} p_{t_1}(a_1)p_{t_1}(a_1-i)p_{2t_2-t_1}(j-a_1)\\
& \leq Cn^{2\alpha+2\epsilon}\sum_{0=i}^n \sum_{(a_1,t_1)\in\Z^2} p_{t_1}(a_1)p_{t_1}(a_1-i)\\
& \leq Cn^{3\alpha+3\epsilon}.
\end{align*}

The second type of term, depicted in (B) of figure 3, comes from when the four points first coalesce in pairs and then the paired ancestral lines coalesce. Say the ancestral line of $(0,0)$ merges with that of $(i,0)$ at $(a_1,t_1)$ and the ancestral line of $(j,0)$ merges with that of $(k,0)$ at $(a_2,t_2)$. Then finally the paired lines merge. This is upper bounded by 
\begin{align*}
C\sum_{(a_1,t_1), (a_2,t_2)\in\Z^2} p_{t_1}(a_1)p_{t_1}(a_1-i)p_{t_2}(a_2-j)p_{t_2}(a_2-k)\P((a_2,t_2)\sim (a_1,t_1)).
\end{align*}
Now summing over $0\leq i,j,k \leq n$ and changing the variables of summation gives
\begin{align*}
C\sum_{0=i,j,k}^n & \sum_{(a_1,t_1), (a_2,t_2)\in\Z^2} p_{t_1}(a_1)p_{t_1}(a_1-i)p_{t_2}(a_2-j)p_{t_2}(a_2-k)\P((a_2,t_2)\sim (a_1,t_1))\\
& = C\sum_{0=i,j,k}^n\sum_{(a_1,t_1), (b_2,t_2)\in\Z^2} p_{t_1}(a_1)p_{t_1}(a_1-i)p_{t_2}(b_2)p_{t_2}(b_2-(k-j))\P((b_2+j,t_2)\sim (a_1,t_1))\\
& \leq C\sum_{0=i,j}^n\sum_{l=-n}^{n}\sum_{(a_1,t_1), (b_2,t_2)\in\Z^2} p_{t_1}(a_1)p_{t_1}(a_1-i)p_{t_2}(b_2)p_{t_2}(b_2-l)\P((b_2+j,t_2)\sim (a_1,t_1))\\
& \leq Cn^{\alpha+\epsilon} \sum_{0=i}^n\sum_{l=-n}^{n}\sum_{(a_1,t_1), (b_2,t_2)\in\Z^2} p_{t_1}(a_1)p_{t_1}(a_1-i)p_{t_2}(b_2)p_{t_2}(b_2-l).
\end{align*}
For the first equality we made the change of variables $b_2=a_2-j$, for the first inequality we made the change of variables $l=k-j$, and in the last inequality we summed of $j$. The sum now can be handled exactly like the previous sums. 
\end{proof}
\begin{remark}
   Lemma \ref{heatkernel} can be adapted to the setting of the Hammond-Sheffield urn, yielding similar estimates without the assumption of the strong renewal theorem. By the arguments of Section 6 this gives another proof of Gaussianity for the Hammond-Sheffield urn.   
\end{remark} 
\vspace{-4mm}

\section{Gaussianity via Lindeberg Swapping}
Now we turn to the Gaussianity of $S(n,t)$. The approach is to write, assuming $n\geq 0$, 
$$\frac{1}{\sigma_n} (S(n,0) -\mathbb{E}S(n,0))=  \frac{1}{\sigma_n}\sum_{\beta}(B_\beta-\mathbb{E}B_{\beta}) |T_\beta\cap [0,n]\times \{0\}|,$$
and first show each $B_\beta-\mathbb{E}B_{\beta}$ may be ``swapped" for an independent Gaussian, $g_\beta$, without affecting the limiting distribution. 
After this it remains to show
$$\frac{1}{\sigma_n}\sum_{\beta}g_{\beta}|T_{\beta}\cap[0,n] \times\{0\}|$$
converges to a Gaussian. For this we note that conditioned on $\mathcal{G}_{\mu}$
the above is a Gaussian with variance 
$$\frac{1}{\sigma_n^2}\sum_{\beta}|T_{\beta}\cap[0,n]\times \{0\}|^2.$$
Now Gaussianity follows from from using Lemma \ref{cov} to show this variance converges to a deterministic limit. 
The proof below is slightly more complicated since we need to prove $S_n(x,t) := \frac{1}{\sigma_n} (S(\lfloor xn\rfloor ,\lfloor tn^{\alpha}L(n)^{-1}\rfloor) -\mathbb{E}S(\lfloor xn\rfloor ,\lfloor tn^{\alpha}L(n)^{-1}\rfloor))$
is a Gaussian field, which means dealing with finite dimensional distributions. 

\begin{proof}[Proof of Theorem 1.2]
Let $\sigma_n$ be as given in Theorem 1. Theorem 1 follow if we show that the random field $S_n(x,t)$, defined in Theorem \ref{thm: main} converges to a Gaussian random field.
Indeed, by the covariance asymptotics in Lemma \ref{covar} the limit field must be $W_\alpha(x,t)$. Recall that to prove convergence to a Gaussian field it suffices 
to show that for any points $(x_1, t_1), ..., (x_k,t_k)\in \R^2$ and coefficients $c_1,...,c_k\in \R$ the random variable 
$$Y_n := \sum_{i=1}^{k} c_i S_n(x_i,t_i) = \sum_{i=1}^k c_i \frac{1}{\sigma_n}S\left(\lfloor x_i n\rfloor , \lfloor t_i L(n)^{-1}n^\alpha\rfloor\right)$$ 
converges to a Gaussian.

\textbf{Step 1:} We perform the swapping. Fix $n$ and enumerate the connected components which intersect 
$\cup_{i=1}^k [0,\lfloor x_i n\rfloor ]\times\{\lfloor t_i L(n)^{-1}n^{\alpha}\rfloor \}$ as $T_{\beta_1}, T_{\beta_2}..., T_{\beta_N}$.
Note $N$ is random, depending on $\mathcal{G}_\mu$. For each $\beta$, let 
\vspace{-3mm}
\begin{equation*}
    c_\beta := \frac{1}{\sigma_n}\sum_{i=1}^k \sgn(x_i)c_i|T_\beta\cap [0,\lfloor x_i n\rfloor ]\times\{\lfloor t_i L(n)^{-1}n^{\alpha}\rfloor \}|.
\end{equation*}
Abusing notation, we let $B_{\beta}$ denote the mean zero random variable $B_{\beta}-\E B_{\beta}$ so that we have
\vspace{-3mm}
\begin{equation*}
    Y_n = \sum_{j=1}^N c_{\beta_j}B_{\beta_j}.
\end{equation*}
Note that the above holds because we multiplied $c_i$ by $\sgn(x_i)$ in the definition of $c_{\beta}$; see the definition of $S(n,t)$ for $n\leq 0$. 
We also introduce independent mean zero Gaussian random variables $(G_\beta)_{\beta\in I}$ with the same variance as $B_{\beta}$, and define
\vspace{-3mm}
$$Y_n^l = \sum_{i=1}^{l-1} c_{\beta_i}G_{\beta_i} +\sum_{i=l+1}^N c_{\beta_i}B_{\beta_i}.$$
\begin{claim}
    If $Y_n^{N+1}$ converges in distribution, then $Y_n$ converges in distribution to the same limit. 
\end{claim}

Fix a test function $\phi\in C^3(\R)$. Writing $\E_{Z}$ to denote expectation with respect to the sigma algebra generated by a random variable $Z$, we have
\begin{equation}
\label{theswap}
\begin{split}
    |\E\phi(Y_n)- \E\phi(Y_n^{N+1})| 
    & = |\E\sum_{i=1}^N \Big(\phi\left(Y_n^{i} + c_{\beta_i} B_{\beta_i} \right)-\phi\left(Y_n^{i} + c_{\beta_i} G_{\beta_i}\right)\Big)|\\
    & = |\E_{\mathcal{G}_\mu}\sum_{i=1}^N \E_{(G_{\beta_i}),(B_{\beta_i})}\Big(\phi^{'}\left(Y_n^{i})c_{\beta_i} (G_{\beta_i} - B_{\beta_i}\right) + \phi^{''}\left(Y_n^{i})c_{\beta_i}^2 (G_{\beta_i}^2 - B_{\beta_i}^2\right)\\
    &  \qquad\qquad\qquad\qquad\qquad\qquad\qquad\qquad - \phi^{'''}\left(q_i\right)c_{\beta_i}^3G_{\beta_i}^3 + \phi^{'''}\left(\tilde{q}_i\right)c_{\beta_i}^3 B_{\beta_i}^3\Big) |\\
    & \leq ||\phi||_{C^3}\left(\E|G_{\beta_1}|^3 + \E|B_{\beta_1}|^3\right)\E_{\mathcal{G}_\mu}\sum_{i=1}^N c_{\beta_i}^3.
\end{split}
\end{equation}
The second line came from Taylor expanding the difference in the first line to third order, where $q_i$ and $\tilde{q}_i$ 
are some points in the intervals $[Y_n^{i},Y_n^{i}+G_{\beta_i}]$ and $[Y_n^{i},Y_n^{i}+B_{\beta_i}]$ respectively. 
To get the third line we noted that for each $i$, the random variables $G_{\beta_i}$, $B_{\beta_i}$ 
are independent of the random variables $Y_n^i$, and $c_{\beta_i}$ and have matching first and second moments. 
Hence if we take the expectation of the $i$-th term with respect to $G_{\beta_i}$ and $B_{\beta_i}$ the $\phi^{'}$ and $\phi^{''}$ 
terms in the sum vanish. Finally we can estimate $\E\sum_{i=1}^N c_{\beta_i}^3$ by Lemma \ref{momentbounds} and the 
bound $(a_1 + a_2+...+a_k)^3\leq Ck^2(a_1^3 + ... + a_k^3)$ to get
\begin{align*}
\E\sum_{i=1}^N c_{\beta_i}^3 
& \leq \frac{Ck^2}{\sigma_n^3}\E \sum_{j=1}^k \sum_{i=1}^N |T_{\beta_i}\cap [0,\lfloor x_jn\rfloor ]\times \{\lfloor t_j n^\alpha L(n)^{-1}\rfloor\}|^3\\
& = \frac{Ck^2}{\sigma_n^3}\E\sum_{j=1}^k\sum_{i=0}^{\lfloor x_jn\rfloor} |T^{S_j}_{(i,\lfloor t_j L(n)^{-1}n^\alpha \rfloor)}|^2 \\
& \leq \frac{Ck^3}{\sigma_n^3}\Big(\max_{j}|x_jn|\Big)\Big(\max_{j,i}\E|T^{S_j}_{(i,\lfloor t_jn^\alpha L(n)^{-1}\rfloor)}|^2\Big) \\
& \leq \frac{C}{\sigma_n^3}n^{1+2\alpha+\epsilon},
\end{align*}
where $S_j = [0,x_jn]\times\{t_jn^{\alpha}L(n)^{-1}\}$. In the last line we used Lemma \ref{momentbounds} with some $\epsilon>0$ small to be chosen. Note by Lemma \ref{momentbounds} the constant $C$ depends on $\epsilon$, $k$, and the $x_i$. 
Finally, since $\sigma_n\geq c n^{\frac{1}{2}(1+\alpha)}L(n)^{-1/2}$, the above estimate and (\ref{theswap}) gives
$$ |\E\phi(Y_n)- \E\phi(Y_n^{N+1})|\leq Cn^{-\frac{1}{2}(1-\alpha-\epsilon)}L(n)^{3/2}.$$
Recall that by Theorem \ref{thm: Potter's}, Potter's Theorem, $L(n)$ grows slower than any power. 
Hence taking $\epsilon$ small enough and $n\to\infty$ proves the claim. Now we show $Y_n^{N+1}$ converges in distribution to a Gaussian.

\textbf{Step 2:} Gaussianity of $Y_n^{N+1}$. Conditioned on the random graph $\mathcal{G}_\mu$, 
the random variable $Y_n^{N+1}$ is a Gaussian with variance $V_n:= \sum_{i=1}^N c_{\beta_i}^2$. 
Thus to prove $Y_n^{N+1}$ converges to a Gaussian it is enough to show that the random variable $V_n$ converges 
to a constant in probability. Estimating the variance using Lemma \ref{cov} and Lemma \ref{momentbounds} gives
\begin{align*}
Var(V_n)  
& = \frac{1}{\sigma_n^4}Var\Big(\sum_{j=1}^N |T_\beta\cap S|^2\Big)\\ 
& = \frac{1}{\sigma_n^4}Var\Big(\sum_{q\in \cup S_j} |T_q^S|\Big)\\
& = \frac{1}{\sigma_n^4} \Big(\sum_{q\in \cup S_j} Var\Big(|T_q^S|\Big) + 2\sum_{q,q_0\in \cup S_j} \Cov\Big(|T_q^S|, |T_{q_0}^S|\Big)\Big)\\
& \leq C n^{-2-2\alpha}\left(n^{1+2\alpha + \epsilon} + n^{1+3\alpha + \epsilon}\right)L(n)^4 \\
& \leq C n^{\alpha+\epsilon-1}L(n)^4,
\end{align*}
where $C$ depends on $x_i$, $t_i$, $\epsilon$, and $\mu$. Since $\alpha <1$ and again Theorem \ref{thm: Potter's} implies $L(n)$ grows slower than any power, we may take $\epsilon$ sufficiently small and conclude $Var(V_n)\to 0$ as $n\to\infty$. Hence $Y_n^{N+1}$ converges to a Gaussian.
\end{proof}

\section{Proof of Theorem \ref{thm:fbm}}
In this section we work with a slightly modified $S_n(\cdot,0):\mathbb{R}\to\mathbb{R}$ 
defined as
$$S_n(x,0) := \frac{1}{\sigma_n}\Big( (1-\{xn\})S(\lfloor xn\rfloor, 0) + \{xn\} S(\lfloor xn \rfloor + 1, 0)-(2p-1)xn\Big),$$
where $\{x\}:= x- \lfloor x \rfloor$ is the fractional part of $x$. In words $S_n$ 
linearly interpolates between integer points. Note that such a modification changes the value of $S_n$ by at most 
$\frac{C}{\sigma_n}$, so it suffices to prove Theorem \ref{cor1} for this modified $S_n$. The advantage of doing this is that 
now $S_n$ is continuous. In fact, the lemma below implies $S_n(x,0)$ is actually Holder continuous. 
This gives the tightness needed to upgrade 
the finite dimensional distributional convergence in Theorem \ref{thm: main}. 

\begin{lemma}
Let $\alpha,p\in(0,1)$ and $\mu\in \Gamma_\alpha$. Then for  any $\gamma<\alpha/2$ we have 
$$\P( ||S_n(\cdot,0)||_{C^\gamma([0,1])} \geq \lambda )\leq C \lambda^{-(\alpha - 2\gamma)\gamma}$$
for all $n, \lambda >0$. The constant $C$ depends on $\mu$ and $\gamma$.
\end{lemma}

Recall that given $f:[0,1]\to\mathbb{R}$ the $C^\gamma$ norm is defined as 
$$||f||_{C^\gamma([0,1])} := \sup_{x,y\in[0,1]}\frac{|f(x)-f(y)|}{|x-y|^\gamma}$$

\begin{proof}
For each $k\geq 0$ we define the set of dyadic points of scale $k$ to be 
\begin{equation}
\label{dyadic}
Q_k := 2^{-k}\mathbb{Z}\cap[0,1],
\end{equation} 
and for any $f:[0,1]\to\mathbb{R}$ we define 
$$D_k(f) := \sup_{p,q\in Q_k, |p-q| = 2^{-k}} |f(p)-f(q)|$$
Note that by applying Lemma \ref{covar} with $x_1 = x_2 = 1$ and $t=0$ we 
have that
$$\mathbb{E} |S(n,0)|^2\leq C \sigma_n^2,$$
for all $n$, where $C$ depends on $\mu$. Hence for any $k,n>0$ we can estimate
\begin{align*}
\P(D_k(S_n(\cdot,0)) \geq 2^{-k\gamma})
& \leq 2^k\P(|S(\lfloor n2^{-k}\rfloor,0)|\geq \sigma_n 2^{-k\gamma})\\
& \leq 2^k \frac{\mathbb{E}|S(\lfloor n2^{-k}\rfloor,0)|^2}{\sigma_n^2}2^{2k\gamma}\\
& \leq C 2^{k(1+2\gamma)}\frac{\sigma_{n 2^{-k}}^2}{\sigma_n^2}\\
& \leq C 2^{k(2\gamma-\alpha)}(\frac{L(n)}{L(n 2^{-k})})^2
\end{align*}
where used Chebyshev's inequality to pass to the second line, the above bound on the 
second moment to pass to the third line, and the definition of $\sigma_n$ to pass to the last line. 
Applying Potter's Theorem, Theorem \ref{thm: Potter's}, to $L(x)$ implies that
there exists a $C>1$ and $m_0>0$, depending on $\gamma$,  such that 
\begin{equation}
    \label{eq: difference estimate}
\P(D_k(S_n(\cdot,0)) \geq 2^{-k\gamma})\leq C  2^{k(\gamma-\alpha/2)}
\end{equation}
for all $k,n\geq 1$ such that  $n 2^{-k}> m_0$. Since $m_0$ does not depend on $n$ or 
$k>0$ and we have the bound $|S(n,0)|\leq n$ almost surely, we may increase
$C$ so that (\ref{eq: difference estimate}) holds for all $n,k\geq 1$.

Now the Lemma follows since by (\ref{eq: difference estimate}) we have 
$$\P(\cap_{k\geq k_0}\{D_k(S_n(\cdot,0))\leq 2^{-k\gamma}\})\leq C 2^{-k_0(\gamma-\alpha/2)},$$
for all $k_0\geq 1$. Indeed, a standard chaining argument implies that we have $||S_n(\cdot,0)||_{C^\gamma([0,1])}\leq C2^{k_0(1-\gamma)}$
in the event that $\cap_{k\geq k_0}\{D_k(S_n(\cdot,t))\leq 2^{-k\gamma}\}$.

For completeness we provide a proof of this last claim. Suppose that $f:[0,1]\to \mathbb{R}$ is a continuous function 
such that $f(0) = 0$, and $D_k(f)\leq 2^{-k\gamma}$ for all $k\geq k_0$. 
We show $||f||_{C^\gamma([0,1])}\leq C 2^{k_0(1-\gamma)}$. 

First, note that since $D_{k_0}(f)\leq 2^{-k\gamma}$ we have the bound $D_k(f)\leq C 2^{k_0-k} 2^{-k_0\gamma}$
for all $k<k_0$.
Now take $x,y\in [0,1]$ with $x\neq y$, and let $n_0$ be such that $2^{-n_0-1}\leq |x-y|\leq 2^{-n_0}$. 
There exists sequences of points $x_{n_0+2}, x_{n_0+3},...$ and $y_{n_{0}+2}, y_{n_{0}+3},...$
such that $x_i\to x$, $y_i\to y$, $y_i,x_i\in D_i$, $|x_i-x_{i-1}|, |y_i-y_{i-1}|\leq 2^-i$, and $|x_{n_{0}+2}-y_{n_{0}+2}|\leq 2^{-n_{0}-2}$. 
Since $f$ is continuous we can estimate with the triangle inequality 
\begin{align*}
|f(x)-f(y)|
& \leq \sum_{i = n_0+2}^\infty |f(x_{i+1})-f(x_i)| + |f(x_{n_0+2})- f(y_{n_{0}+2})| + \sum_{i = n_0+2}^{\infty} |f(y_{i+1})-f(y_i)|.\\
\end{align*}
Now if $n_0\geq k_0$ then above sum implies $|f(x)-f(y)|\leq C 2^{-n_0\gamma}\leq C|x-y|^{\gamma}$. On the other hand 
if $n_0\leq k$ the above sum implies $|f(x)-f(y)|\leq C 2^{k_0-n_0}2^{-k_0\gamma}\leq C |x-y|^{\gamma}2^{k_0(1-\gamma)}.$
These two bounds imply the claim. 
\end{proof}

Now Theorem \ref{thm:fbm} follows from the above lemma and the Theorem \ref{thm: main} by a diagonalization 
argument. 
By Theorem \ref{thm: main} 
there are couplings $(\tilde{C}^k_n)_{k,n\in\mathbb{N}}$ such that for each $\epsilon>0$ and $k\in\mathbb{N}$
$$\lim_{n\to\infty}\tilde{C}_n^k(||S_n(\cdot,0)-B^{1/2+\alpha/2}||_{L^\infty(2^{-k}\mathbb{Z}\cap[0,k])}\geq \epsilon) = 0.$$
By taking a subsequence of the couplings we can find couplings $(C_n)_{n\in\mathbb{N}}$ 
such that for each $k\in \mathbb{N}$ and $\epsilon>0$
$$\lim_{n\to\infty}C_n(||S_n(\cdot,0)-B^{1/2+\alpha/2}||_{L^\infty(2^{-k}\mathbb{Z}\cap[0,k])}\geq \epsilon) = 0.$$
We claim that $(C_n)_{n\in\mathbb{N}}$ is our desired coupling. 
Fix $T,\epsilon>0$. Note that for any $f,g\in C^0([0,T])$ 
and $k\in \mathbb{N}$ one has 
$$||f-g||_{L^{\infty}([0,T])}\leq 
||f-g||_{L^{\infty}(2^{-k}\mathbb{Z}\cap [0,T])} +  2^{-k\gamma}(||g||_{C^\gamma([0,T])}+||f||_{C^\gamma([0,T])}).$$
Since $S_n(\cdot,0)$ and $B^{1/2+\alpha/2}$ are continuous almost surely we can estimate for any $n\in \mathbb{N}$ and $k>T$
\begin{align*}
    C_n(||S_n(\cdot,0)- B^{1/2+\alpha/2} ||_{L^\infty([0,T])}\geq \epsilon)
& \leq C_n(||S_n(\cdot,0)-B^{1/2+\alpha/2}||_{L^{\infty}(2^{-k}\mathbb{Z}\cap [0,T])}\geq \epsilon/3)\\
& + C_n(||S_n(\cdot,0)||_{C^\gamma([0,T])}\geq 2^{k\gamma}\epsilon/3)+C_n(||B^{1/2+\alpha/2}||_{C^\gamma([0,T])}\geq 2^{k\gamma}\epsilon/3)\\
& \leq C_n(||S_n(\cdot,0)-B^{1/2+\alpha/2}||_{L^{\infty}(2^{-k}\mathbb{Z}\cap [0,T])}\geq \epsilon/3) + CT (\epsilon 2^{k\gamma})^{-(\alpha-2\gamma)\gamma}.
\end{align*}
We  applied Lemma 6.1 and a union bound to the second
term and used the well known bound on the Holder contuity of fractional Brownian motion for the third term. 
By the construction of the $C_n$ this implies that
$$\limsup_{n\to\infty}C_n(||S_n(\cdot,0)- B^{1/2+\alpha/2} ||_{L^\infty([0,T])}\geq \epsilon)
\leq CT (\epsilon 2^{k\gamma})^{-(\alpha-2\gamma)\gamma}.$$
But $k\in \mathbb{N}$ is arbitrary and $T,\epsilon>0$ are fixed, so taking $0<\gamma <\alpha/2$, and 
$k\to\infty$ proves the theorem. 

\section{Applying Theorem \ref{thm: main} to the Voter Model}
In this section we prove Theorem \ref{cor1}. It follows from 
Theorem \ref{thm: main} and the observation that the random field 
$X_p(n,t):\mathbb{Z}^2\to \{-1,+1\}$
restricted to a time slice $\mathbb{Z}\times \{0\}\subset \mathbb{Z}^2$
has the law of $\nu_p$. Recall that $\nu_p$ is the extremal invariant measure 
of the voter model satisfying 
$$\nu_p(\{X:\mathbb{Z}\to\{-1,+1\}|\ X(0) = 1\})=p.$$
We will use the following characterization of the $\nu_p$ which follows 
immediately from Theorem 1.9 in \cite{holllig}.
\begin{theorem}\label{thm: extremal measures}
    Let $\alpha\in(0,1)$, and $\mu\in\Gamma_\alpha$. 
    The voter model with interactions $\mu$ has a one-parameter 
    family of extremal invariant measures given by $\{\nu_p\ | p\in[0,1]\}$, where $\nu_p$ 
    is the unique extremal invariant measure with 
    $\nu_p(\{\eta\ | \eta(0) = 1\}) = p$. Moreover, if 
    $(\eta_t)_{t\in\N}$ is the evolution of the voter model from 
    initial state $\eta_0$, distributed according to the $p$-Bernoulli product measure,
    then the law of $\eta_t$ converges to $\nu_p$.
\end{theorem}
Now we prove Theorem \ref{cor1}.
\begin{proof}[Proof of Theorem \ref{cor1}]
First, we show that the field $X_p(n,t)$,restricted to $\mathbb{Z}\times\{0\}$, has the law $\nu_p$. 
After this we show $X_p(\cdot,0)$ converges to fractional Gaussian noise by Theorem \ref{thm: main}.
Fix $p\in(0,1)$ and let $\eta_0:\mathbb{Z}\to\{-1,+1\}$ be distributed according to the 
$p$-Bernoulli product measure. Define $(\eta_t)_{t\in\N}$ to be the voter model 
evolving from $\eta_0$. We argue that $X_p(\cdot,t)$ and $\eta_t$ have the same 
limiting distribution. Since Theorem \ref{thm: extremal measures} implies that 
$\eta_t$ has $\nu_p$ as a limiting distribution, and $X_p(\cdot,t)$ has the same 
distribution for any $t$ we can conclude $X_p(\cdot,0)$ has law $\nu_p$. 

Fix any $S := \{q_1,q_2,...,q_m\}\subset \Z$ and $t\in \mathbb{N}$. 
We construct a coupling $(\xi,\eta)$ such that 
$\eta,\xi:S\to\{-1,+1\}$ have the same law as $\eta_t$ and $X_p(\cdot,t)$ restricted to the 
set $S$, and 
$\xi(q_i) = \eta(q_i)$ for all $i = 1,...,m$ with probability tending to 1 as $t\to\infty$.

To construct the coupling, sample $G_\mu$, and for any $(n,t)\in\mathbb{Z}^2$ and $s\leq t$ define $a_s(n,t)$ to be 
the unique ancestor, in $G_\mu$, of $(n,t)$ at time $s$. We define $\eta:\mathbb{Z}\to\{-1,+1\}$ 
simply by 
$$\eta(q_i) := \eta_0(a_0(q_i,t)).$$
Note $\eta$ has the law of $\eta_t$ restricted to the set $S$. 
To construct $\xi$ we define the time 
$$t_{f} := \sup \{s\leq t\ |\ |\{a_s(q_i,t)\ |\ i\in [1,m]\}|=|\{a_{s_0}(q_i,t)\ |\ i\in [1,m]\}|\ \forall\ s_0\leq s\}.$$
In words, $t_f$ is the last time that there is any coalescence among 
the ancestral lines of $(q_1,t),...,(q_m,t)$. Now, define
$$\xi(m) = \eta_0(a_{\min\{t_f,0\}}(n,t)).$$
By the definition of $X_p(n,t)$, $\xi$ has the same law as $X_p(\cdot,t)$ restricted to $S$. 
Now we note that in the event $\{t_f\geq 0\}$ we have that $\xi(q_i) = \eta(q_i)$ for each $i\in [1,m].$ 
But for any fixed finite $S$,
$$\P\left(\{t_f\geq 0\}\right)\to 1,$$
as $t\to\infty,$ so we see that $\eta_t$ and $X_p(\cdot,t)$ have the same limit law. 
Now applying Theorem \ref{thm: extremal measures}, and using that $X_p(\cdot,t)$ has the same 
distribution for all $t$, we see that $X_p(\cdot,0)$ has law $\nu_p$.

By the above observation it now suffices to prove that if $\xi(n) = X_p(n,t)$ and $F^n_p:\mathcal{S}\to\mathbb{R}$ is the functional 
defined by 
$$F^n_p(\phi) := \frac{1}{\sigma_n\sqrt{\alpha(\alpha/2+1/2)}}\sum_{i\in\mathbb{Z}} \phi(i/n)(\xi(i)-(2p-1)),$$
then $F_p^n(\phi)$ converges in distribution to a mean $0$ Gaussian with variance 
$$B_{\alpha/2}(\phi,\phi):= \int_{\mathbb{R}}\int_{\mathbb{R}} \frac{\phi(x)\phi(y)}{|x-y|^{1-\alpha}}.$$
We do this by applying Theorem \ref{thm: main}
 
Fix $\phi\in C_c^\infty(\R)$ with $\supp(\phi)\subset B_R(0)$. For any $N>0$ let $\round{x}$ 
denote the closest element of $\frac{1}{N}\Z$ to $x$ and in what follows we let 
$$c_n:= \frac{1}{\sigma_n\sqrt{\alpha(\alpha/2+1/2)}}.$$
For any $N\in\mathbb{N}$ we compute
\begin{equation}
\begin{split}
F^n_p(\phi)
& = c_n\sum_{i\in\Z} \phi\left(\frac{i}{n}\right)\left(\xi(i)-(2p-1)\right)\\
& = c_n\sum_{i\in \Z\cap B_{Rn}(0)}\phi\left(\round{\frac{i}{n}}\right)\left(S_n\left(\round{\frac{i}{n}} + \frac{1}{N}\right)-S_n\left(\round{\frac{i}{n}}\right)\right)\\ 
& \quad\quad\quad\quad\quad\quad + c_n\sum_{i\in \Z\cap B_{Rn}(0)} \left(\phi\left(\frac{i}{n}\right)-\phi\left(\round{\frac{i}{n}}\right)\right)(\xi(i)-(2p-1))\\
& := M_N(n) + E_N(n).
\end{split}
\end{equation}
$S_n$ is as defined in Theorem \ref{thm: main}. By Theorem \ref{thm: main}, if we fix $N$,
$M_N(n)$ converges in distribution to a centered Gaussian with variance
\begin{equation*}
\begin{split}
\frac{1}{\alpha(\alpha/2+1/2)}& \sum_{x,y\in \frac{1}{N}\Z}\phi(x)\phi(y)\E\left(\left(B^{\frac{1+\alpha}{2}}\left(x+\frac{1}{N}\right)- B^{\frac{1+\alpha}{2}}(x)\right)\left(B^{\frac{1+\alpha}{2}}\left(y+\frac{1}{N}\right)- B^{\frac{1+\alpha}{2}}(y)\right)\right)\\
& = \frac{1}{\alpha(\alpha+1)}\sum_{x,y\in\frac{1}{N}\Z}\phi(x)\phi(y)\left(|x-y+\frac{1}{N}|^{2H}+|x-y-\frac{1}{N}|^{2H}-2|x-y|^{2H}\right)\\
& = \sum_{x\neq y\in\frac{1}{N}\Z}\phi(x)\phi(y)|x-y|^{2H-2}N^{-2} + O(RN^{-\alpha}), 
\end{split}
\end{equation*}
as $n\to\infty$. Recall $B^{\frac{1+\alpha}{2}}$ is fractional Brownian motion with Hurst parameter $\frac{1}{2} + \frac{\alpha}{2}$, and 
we used that $\E B^H(t)B^H(s) = \frac{1}{2}(|t|^{2H} + |s|^{2H} - |t-s|^{2H})$ 
in the first equality, and then Taylor expanded for the second equality. For any $N\in \mathbb{N}$ we also have the estimate
\begin{align*}
\mathbb{E}|E_N(n)|^2
& \leq c_n^2\sum_{i,j\in B_{Rn}(0)} \left(\phi\left(\frac{i}{n}\right)-\phi\left(\round{\frac{i}{n}}\right)\right)\left(\phi\left(\frac{j}{n}\right)-\phi\left(\round{\frac{j}{n}}\right)\right)\Cov(\xi(i),\xi(j))\\
& \leq c_{n}^2 \frac{||\phi||_{Lip}^2}{N^2}\sum_{i,j\in B_{Rn}(0)}|\Cov(\xi(i),\xi(j))|\\
& \leq \frac{C||\phi||_{Lip}^2}{\sigma_n^2 N^2}\sum_{i\in B_{Rn}(0)}\mathbb{E}|T_{(i,0)}\cap B_{Rn}(0)\times\{0\}|\\
& \leq \frac{C||\phi||_{Lip}^2}{\sigma_n^2 N^2}\sigma_{2Rn}^2\\
& \leq \frac{C||\phi||_{Lip}^2R^{1+\alpha}}{N^2},
\end{align*}
where $C$ depends on $\alpha$, where we applied Lemma \ref{covar} to bound the sum of the covariances. Hence for any fixed $\phi$ 
we may take a sequence $(N_n)_{n\in\N}$ such that $N_n\to\infty$ as $n\to\infty$ sufficiently slowly, and see that 
$F^n_p(\phi)$ converges in distribution to the centered Gaussian with variance 
$$\int_{\R}\int_{\R}\frac{\phi(x)\phi(y)}{|x-y|^{1-\alpha}}dxdy.$$
To see this for a general Schwarz function we approximate it by a compactly supported function and use the following variance estimate. For any $\phi\in\mathcal{S}(\R)$ compute
\begin{equation}
\begin{split}
Var(F_p^n(\phi))
& = \frac{1}{\alpha(\alpha/2 + 1/2)\sigma_n^2}\E (\sum_{i\in\Z}\phi\left(\frac{i}{n}\right)(\xi(i)-(2p-1)))^2\\
& \leq \frac{C}{\sigma_n^2}\sum_{i,j\in\Z}|\phi\left(\frac{i}{n}\right)||\phi\left(\frac{j}{n}\right)|\Cov (\xi(i),\xi(j))|\\
& \leq \frac{C}{\sigma_n^{2}}\sum_{l\in\N}\left(\max_{(2^{l-1}-1)n\leq k\leq 2^l n}\sum_{i\in \Z}\phi\left(\frac{i}{n}\right)\phi\left(\frac{i}{n} + \frac{k}{n}\right)\right)\sum_{2^{l-1}n\leq k\leq 2^l n}\P((0,0)\sim (0,k))\\
& \leq Cn^{-1-\alpha}L(n)\sum_{l\in\N} (n||\phi||_{L^1}||x^2\phi(x)||_{L^\infty} 2^{-l}) 2^{l\alpha}n^{\alpha}L(2^ln)^{-1}\\
& \leq C ||\phi||_{L^1}||x^2\phi(x)||_{L^\infty} \sum_{l\in\N}2^{-(1-\alpha)l}L(n)L(2^ln)^{-1}\\
& \leq C ||\phi||_{L^1}||x^2\phi(x)||_{L^\infty},
\end{split}
\end{equation}
where we used Lemma \ref{covar} and that $\sigma_n^2 = \tilde{c}_p n^{1+\alpha}L(n)^{-1}$ 
in the third inequality, and used Potter's Theorem for slowly varying functions 
in the last line. Here $C$ depends on $\mu$ through its dependence on $\alpha$. 
Now for $\psi\in\mathcal{S}(\R)$ take a sequence of lengths $R_n\to\infty$ and
write $\psi = \psi \chi(x/ R_n) + \psi (1-\chi(x/R_n))$ for a smooth bump function 
$\chi$ supported in the ball of radius $1$. Taking $R_n\to\infty$ slowly enough in 
$n$ gives that $F_p^n(\psi\chi(\cdot/R_n))$ converges to $ FGF_{\alpha/2}(\psi)$ in 
distribution, by the result for compactly supported functions, and that 
$F_p^n(\psi(1-\chi(\cdot/R_n))$ converges to $0$ in probability, by the variance 
estimate above. 

The above implies that $F_p^n$, viewed as a random tempered distribution, converges in distribution to $FGF_{\alpha/2}$.
A proof of this is given in \cite{generalizedrandomfields}; see Corollary 2.4. For the sake of the reader we explain 
why this is true. In this setup convergence in distribution means that given any bounded continuous function 
$\Phi: (\mathcal{S}(\mathbb{R}))' \to \mathbb{R}$ 
one has
\begin{equation}
\label{eq:convergence in distribution}
\mathbb{E} \Phi(F_p^n)\to \mathbb{E} \Phi(FGF_{\alpha/2}),
\end{equation}
as $n\to\infty$. Here we have given $(\mathcal{S}(\mathbb{R}))'$ the strong topology, and 
so, similar to the situation of weak convergence of probability measure on $\mathbb{R}$,
\eqref{eq:convergence in distribution} follows if for any $\psi\in\mathcal{S}(\mathbb{R})$
$$\mathbb{E}e^{iF_p^n(\psi)}\to \mathbb{E}e^{iFGF_{\alpha/2}(\psi)},$$
as $n\to\infty$.
But this follows from the above result on the convergence in distribution of $F_p^n(\psi)$.  

\end{proof}

\printbibliography

\end{document}